\documentclass[11pt]{amsart} 

\usepackage[utf8]{inputenc} 
\usepackage{tikz}
\usepackage{mathtools}
\usepackage{extpfeil}
\usepackage{graphicx}
\usepackage{gensymb}

\title{Geometric Realizations of the Basic Representation of $\widehat\gl_r$}
\author{Joel Lemay}

\usepackage{amsmath,amsfonts,amssymb,amsthm,amscd,comment,textcomp}
\usepackage[colorlinks=true, pdfstartview=FitV, linkcolor=blue, citecolor=blue, urlcolor=blue, breaklinks=true]{hyperref}

\theoremstyle{plain}
\newtheorem{thm}{Theorem}[section]
\newtheorem{prop}[thm]{Proposition}
\newtheorem{lemma}[thm]{Lemma}
\newtheorem{cor}[thm]{Corollary}

\theoremstyle{definition}
\newtheorem{defn}[thm]{Definition}
\newtheorem{rmk}[thm]{Remark}

\newtheorem*{ack}{Acknowledgements}

\numberwithin{equation}{section}

\newcommand{\Z}{\mathbb{Z}}
\newcommand{\C}{\mathbb{C}}
\DeclareMathOperator{\Cl}{Cl}
\newcommand{\M}{\mathcal{M}}
\newcommand{\N}{\mathbb{N}}
\DeclareMathOperator{\Hom}{Hom}
\DeclareMathOperator{\GL}{GL}
\DeclareMathOperator{\st}{st}
\DeclareMathOperator{\im}{im}
\newcommand{\QQ}{\widetilde{Q}}
\newcommand{\MM}{\mathfrak{M}}
\DeclareMathOperator{\id}{id}
\DeclareMathOperator{\tr}{tr}
\DeclareMathOperator{\vspan}{span}

\DeclareMathOperator{\rk}{rk}

\newcommand{\OO}{\mathcal{O}}
\newcommand{\vv}{\mathbf{v}}
\newcommand{\ww}{\mathbf{w}}
\newcommand{\cc}{\mathbf{c}}
\newcommand{\T}{\mathcal{T}}
\newcommand{\dd}{\mathrm{d}}

\newcommand{\V}{\mathcal{V}}
\newcommand{\W}{\mathcal{W}}

\newcommand{\HH}{\mathcal{H}}
\DeclareMathOperator{\pt}{pt}
\newcommand{\dee}{\mathbf{d}}
\newcommand{\nn}{\mathbf{n}}
\newcommand{\A}{\mathbf{A}}
\newcommand{\K}{\mathcal{K}}
\DeclareMathOperator{\tnv}{tnv}
\newcommand{\geoP}{\mathbf{P}}

\newcommand{\uu}{\mathbf{u}}
\newcommand{\KK}{\mathfrak{K}}
\newcommand{\II}{\mathbf{I}}
\newcommand{\JJ}{\mathbf{J}}
\newcommand{\mm}{\mathbf{m}}
\newcommand{\1}{\mathbf{1}}
\newcommand{\geoE}{\mathbf{E}}
\newcommand{\geoF}{\mathbf{F}}
\newcommand{\geoH}{\mathbf{H}}
\newcommand{\sla}{\mathfrak{sl}}
\newcommand{\g}{\mathfrak{g}}

\newcommand{\geoPsi}{\mathbf{\Psi}}
\newcommand{\frakE}{\mathfrak{E}}
\newcommand{\frakF}{\mathfrak{F}}
\newcommand{\frakH}{\mathfrak{H}}
\newcommand{\bK}{\mathbf{K}}
\newcommand{\bL}{\mathbf{L}}
\newcommand{\FF}{\mathbb{F}}

\newcommand{\gl}{\mathfrak{gl}}

\newcommand{\osc}{\mathfrak{s}}
\newcommand{\llbracket}{\text{\textlbrackdbl}}
\newcommand{\rrbracket}{\text{\textrbrackdbl}}
\newcommand{\BB}{\mathbb{B}}

\DeclareMathOperator{\lcm}{lcm}

\subjclass[2010]{17B65, 14F43, 05E10}

\thanks{This research was supported by the NSERC Discovery Grant of Alistair Savage and an NSERC PGS-D scholarship}

\begin{document}
\maketitle

\begin{abstract}
The realizations of the basic representation of $\widehat\gl_r$ are known to be parametrized by partitions of $r$ and have an explicit description in terms of vertex operators on the bosonic/fermionic Fock space. In this paper, we give a geometric interpretation of these realizations in terms of geometric operators acting on the equivariant cohomology of certain Nakajima quiver varieties.
\end{abstract}

\tableofcontents

\section*{Introduction} \thispagestyle{empty}

Let $\g$ be a semisimple Lie algebra and denote the corresponding untwisted affine Lie algebra by $\widehat\g$. The basic representation of $\widehat\g$, which we denote by $V_\text{basic} = V_\text{basic}(\widehat\g)$, is the irreducible highest weight representation whose highest weight is the fundamental weight corresponding to the additional node of the affine Dynkin diagram (compared to the corresponding finite Dynkin diagram). The basic representation is so-named since it is, in a sense, the simplest representation of $\widehat\g$. In the late 70's and early 80's mathematicians began constructing explicit realizations of $V_\text{basic}$. The first such realization was given by Lepowsky and Wilson in \cite{lepwil} for $V_\text{basic}(\widehat\sla_2)$. Their construction was later generalized to arbitrary simply-laced affine Lie algebras and twisted affine Lie algebras in \cite{kklw}, and this construction became known as the \emph{principal} realization of $V_\text{basic}$. However, Frenkel and Kac in \cite{frekac}, and Segal in \cite{segal}, gave an entirely different realization of $V_\text{basic}$; this construction was referred to as the \emph{homogeneous} realization. While the principal and homogeneous realizations seemed completely unrelated, it was discovered by Kac and Peterson in \cite{kacpet}, and by Lepowsky in \cite{lep}, that the two realizations depend implicitly on the choice of a so-called maximal \emph{Heisenberg} subalgebra of $\widehat\g$. Indeed, one can associate a realization of $V_\text{basic}$ to each maximal Heisenberg subalgebra of $\widehat\g$.

In this paper will focus on the case where $\g = \gl_r$. While $\gl_r$ is not semisimple, it is a one-dimensional central extension of the semisimple Lie algebra $\sla_r$, and thus has a similar representation theory. Up to conjugacy under the adjoint action of the Kac-Moody group, the maximal Heisenberg subalgebras of a semisimple affine Lie algebra are known to be parametrized by the conjugacy classes of the Weyl group of the corresponding finite-dimensional Lie algebra (see \cite[Proposition of Section 9]{kacpet}). The Weyl group of $\sla_r$ and $\gl_r$ is the symmetric group on $r$ elements, $S_r$, and the conjugacy classes of $S_r$ are in one-to-one correspondence with \emph{partitions} of $r$, i.e.\ $s$-tuples $(r_1,\dots,r_s) \in (\N^+)^s$ such that
\[\textstyle{ r = r_1 + \dots + r_s, \quad \text{and} \quad r_1 \le \dots \le r_s.}\]
Thus, there exists a realization of $V_\text{basic}(\widehat\gl_r)$ for each partition of $r$, the principal and homogeneous realizations corresponding to the two extreme partitions $(r)$ and $(1,\dots,1)$, respectively. The realizations for every partition of $r$ were described by ten Kroode and van de Leur in \cite{tkvdl} using vertex operators acting on bosonic Fock space (a representation of the Heisenberg algebra) and fermionic Fock space (a representation of the Clifford algebra). More precisely, for each partition of $r$, there exists a precise vector space isomorphism between bosonic Fock space and fermionic Fock space (known as the boson-fermion correspondence), and thus the Heisenberg and Clifford algebras may be thought of as operators acting on a common space. The construction in \cite{tkvdl} defines a representation of $\widehat\gl_r$ on bosonic/fermionic Fock space in terms of vertex operators (i.e.\ formal power series) of Heisenberg and Clifford algebra operators. The so-called ``zero-charge" subspace of bosonic/fermionic Fock space is then shown to be isomorphic, as a $\widehat\gl_r$-representation, to $V_\text{basic}$.

In this paper, we give a geometric interpretation of these algebraic realizations of $V_\text{basic}(\widehat\gl_r)$. Our general strategy is as follows. We fix a partition of $(r_1,\dots,r_s)$ of $r$ and consider the moduli space of framed torsion-free sheaves of rank $s$ and second Chern class $n$, $\M(s,n)$. In \cite{savlic}, Licata and Savage showed that, under a suitable torus action, the localized equivariant cohomology of (a disjoint union of infinitely-many copies) of $\M(s,n)$ provides a suitable geometric analogue of bosonic/fermionic Fock space. This is accomplished by defining an action of the Heisenberg and Clifford algebras on this cohomology in terms of the top Chern classes of certain equivariant vector bundles on $\M(s,n)$. The construction given in \cite{savlic} naturally corresponds to the homogeneous realization in \cite{tkvdl}, and thus our first step is to generalize their construction to an arbitrary partition. With this framework in place, we define a new set of operators using vector bundles on certain subvarieties of $\M(s,n)$. We then show that these operators may be expressed as vertex operators of our ``geometric" Heisenberg and Clifford algebra operators, and that the formulas we obtain exactly match those found in the algebraic realization of the action of $\widehat\gl_r$ on $V_\text{basic}$, thus giving us a geometric realization of $V_\text{basic}$

The paper is organized into 4 sections. In Section \ref{section:1}, we review the Heisenberg and Clifford algebra representations on bosonic and fermionic Fock space. We also briefly summarize the various algebraic realizations of $V_\text{basic}$ found in \cite{tkvdl}. In Section \ref{section:2}, we review some of the basic geometric concepts that we will use in subsequent sections. In particular, we will discuss our main geometric object of interest: Nakajima quiver varieties (of which the aforementioned moduli space $\M(s,n)$ is a special case). Section \ref{section:3} will describe our method of constructing geometric operators on the localized equivariant cohomology of quiver varieties from equivariant vector bundles. Finally, in Section \ref{section:4}, we define our geometric analogues of the action of the Heisenberg algebra, Clifford algebra, and $\widehat\gl_r$ on bosonic and fermionic Fock space. We also present our main theorem (Theorem \ref{thm:main}), which is a geometric analogue of Proposition \ref{thm:tkvdl} (the main theorem of \cite{tkvdl}).

\section{The Basic Representation \texorpdfstring{of $\widehat\gl_r$}{}}\label{section:1}

In this first section, we will briefly summarize the known algebraic realizations of the basic representation of $\widehat\gl_r$. In particular, the inequivalent realizations are parametrized by the different partitions of $r$. For a more in-depth treatment of this topic, the reader is encouraged to see \cite{tkvdl} or \cite{kac}. The goal in the subsequent sections will be to give a geometric construction of the representations presented here.

We begin by recalling the $s$-coloured oscillator algebra and the $s$-coloured Clifford algebra, where $s \in \N^+$, along with the associated $s$-coloured bosonic and fermionic Fock spaces.

\begin{defn}($s$-coloured oscillator algebra)\label{defn:osc}
The \emph{$s$-coloured oscillator algebra}, $\osc$, is the complex Lie algebra
\[\textstyle{ \osc := \bigoplus_{\ell=1}^s \left( \bigoplus_{n \in \Z} \C P_\ell(n) \right) \oplus \C c,}\]
with the Lie bracket determined by
\[\textstyle{ [\osc,P_\ell(0)] = 0, \quad [P_\ell(n),P_k(m)] = \frac{1}{n} \delta_{\ell,k} \delta_{n+m,0} c, \; \; n \ne 0,}\]
for all $\ell,k = 1,\dots, s$ and $m,n \in \Z$.
\end{defn}

Note that the $s$-coloured oscillator algebra is often presented in terms of the basis $\{ \alpha_\ell(n) \}$, where
\[ \textstyle{ \alpha_\ell(n) := |n| P_\ell(n), \; n \ne 0, \qquad \alpha_\ell(0) := P_\ell(0).}\]

\begin{defn}($s$-coloured Heisenberg algebra)
The subalgebra
\[\textstyle{ \osc_0 = \bigoplus_{\ell=1}^s \left( \bigoplus_{n \in \Z-\{0\}} \C P_\ell(n)\right) \oplus \C c,}\]
of $\osc$ is the \emph{$s$-coloured Heisenberg algebra}.
\end{defn}

Let $\Lambda \subseteq \C \llbracket t_1, t_2, \dots \rrbracket$ denote the ring of symmetric functions in infinitely many variables. It is well-known that
\[\textstyle{ \Lambda = \C[p_1, p_2, \dots], }\]
where $p_n$ is the $n$-th power sum
\[\textstyle{ p_n = \sum_{i=1}^\infty t_i^n. }\]
Define \emph{$s$-coloured bosonic Fock space} to be the space
\[\textstyle{ \BB := B^{\otimes s}, \quad \text{where } B := \Lambda \otimes_{\C} \C[q,q^{-1}]. }\]
We have a $\Z$-grading on $B$ given by
\[\textstyle{ B = \bigoplus_{c \in \Z} B_c, \quad B_c := \Lambda \otimes \C q^c.}\]
This induces a $\Z^s$-grading on $\BB$ given by
\[\textstyle{ \BB = \bigoplus_{\cc \in \Z^s} \BB_{\cc}, \quad \BB_{\cc} := B_{\cc_1} \otimes \cdots \otimes B_{\cc_s}.}\]
For $\cc = (\cc_1, \dots, \cc_s) \in \Z^s$, we use the notation $|\cc| = \cc_1 + \dots + \cc_s$. We then have a $\Z$-grading on $\BB$ given by
\[\textstyle{ \BB = \bigoplus_{c \in \Z} \BB(c), \quad \BB(c) = \bigoplus_{|\cc|=c} \BB_{\cc}.}\]
One can easily verify that the mapping
\[\textstyle{ P_\ell(n) \mapsto 1^{\otimes \ell-1} \otimes \frac{\partial}{\partial p_n} \otimes 1^{\otimes s - \ell}, \quad n > 0,}\]
\[\textstyle{ P_\ell(-n) \mapsto 1^{\otimes \ell-1} \otimes \frac{1}{n} p_n \otimes 1^{\otimes s-\ell}, \quad n > 0,}\]
\[\textstyle{ P_\ell(0) \mapsto 1^{\otimes \ell-1} \otimes q \frac{\partial}{\partial q} \otimes 1^{\otimes s-\ell}, \quad c \mapsto 1, }\]
defines an irreducible representation of $\osc$ on $\BB$. 

\begin{rmk}\label{rmk:colours}
Notice that the $s$-coloured Heisenberg algebra is isomorphic to the $1$-coloured Heisenberg algebra via relabelling of the indices. Thus, we will often use the term ``Heisenberg algebra" without specifying the number of colours. Moreover, one can show (see, for instance, \cite[Lemma 14.4(a)]{kac}) that $\Lambda^{\otimes s}$ and $\Lambda$ are isomorphic as Heisenberg algebra modules.
\end{rmk}

\begin{defn}($s$-coloured Clifford algebra)
The \emph{$s$-coloured Clifford algebra}, $\Cl$, is the complex associative algebra generated by elements $\psi_\ell(i), \psi^*_\ell(i)$, $\ell=1,\dots,s$, $i \in \Z$, and the relations
\[\textstyle{ \{\psi_\ell(i),\psi^*_k(j)\} = \delta_{ij}\delta_{\ell k}, \quad  \{ \psi_\ell(i), \psi_k(j) \} = \{ \psi_\ell^*(i),\psi^*_k(j) \} = 0.}\]
where $\{x,y\} = xy + yx$.
\end{defn}

A \emph{semi-infinite monomial} is an expression of the form
\[\textstyle{ I = i_1 \wedge i_2 \wedge i_3 \wedge \cdots,}\]
where the $i_n$ are integers such that
\[\textstyle{ i_1 > i_2 > i_3 > \dots, \quad i_{n+1} = i_n - 1, \text{ for } n \gg 0.}\]
The \emph{charge} of a semi-infinite monomial $I$ is the integer $c=c(I)$ such that
\[\textstyle{ i_n = c-n+1, \text{ for all } n \gg 0.}\]
The \emph{energy} of a semi-infinite monomial is
\[\textstyle{ \sum_{n \in \N^+} i_n - (c-n+1).}\]
Let $F$ be the complex vector space with basis the set of all semi-infinite monomials. The charge induces a grading on $F$:
\[\textstyle{ F = \bigoplus_{c \in \Z} F_c,}\]
where $F_c$ is the subspace with basis the set of semi-infinite monomials of charge $c$. Define \emph{wedging} and \emph{contracting} operators $\psi(i)$ and $\psi^*(i)$, $i \in \Z$, on $F$ by
\[\textstyle{ \psi(i) (i_1 \wedge i_2 \wedge \cdots) = \begin{cases} (-1)^k (i_1 \wedge \cdots \wedge i_k \wedge i \wedge i_{k+1} \wedge \cdots), & \text{if } i_k > i > i_{k+1}, \\ 0, & \text{if } i = i_n, \text{ for some } n, \end{cases} }\]
\[\textstyle{ \psi^*(i) (i_1 \wedge i_2 \wedge \cdots) = \begin{cases} (-1)^{k-1} (i_1 \wedge \cdots \wedge i_{k-1} \wedge i_{k+1} \wedge \cdots), & \text{if } i = i_k, \\ 0, & \text{if } i \ne i_n, \text{ for all } n. \end{cases} }\]
It is easy to see that the wedging and contracting operators raise and lower the charge of a semi-infinite monomial by $1$, and so
\[\textstyle{ \psi(i): F_c \to F_{c+1}, \quad \psi^*(i): F_c \to F_{c-1},}\]
for all $i,c \in \Z$.

Define \emph{$s$-coloured fermionic Fock space} to be
\[\textstyle{ \FF := F^{\otimes s}. }\]
By our decomposition of $F$ in terms of the charge, we have a $\Z^s$-grading
\[\textstyle{ \FF = \bigoplus_{\cc \in \Z^s} \FF_{\cc}, \quad \FF_{\cc} := F_{\cc_1} \otimes \cdots \otimes F_{\cc_s}.}\]
This induces a $\Z$-grading on $\FF$ given by
\[\textstyle{ \FF = \bigoplus_{c \in \Z} \FF(c), \quad \FF(c) := \bigoplus_{|\cc|=c} \FF_{\cc}.}\]
One can show that we have a representation of $\Cl$ on $\FF$ given by:
\[\textstyle{ \psi_\ell(i)|_{\FF_{\cc}} \mapsto (-1)^{\cc_1 + \dots + \cc_{\ell-1}} \left( 1^{\otimes \ell-1} \otimes \psi(i) \otimes 1^{\otimes s-\ell} \right),}\]
\[\textstyle{ \psi_\ell^*(i)|_{\FF_{\cc}} \mapsto (-1)^{\cc_1 + \dots + \cc_{\ell-1}} \left( 1^{\otimes \ell-1} \otimes \psi^*(i) \otimes 1^{\otimes s-\ell} \right).}\]
It is straightforward to show that $\FF$ is an irreducible representation of $\Cl$. Moreover, $\FF$ is generated by the so-called \emph{vacuum vector}, $|0 \rangle^{\otimes s}$, where
\[\textstyle{ |0 \rangle = 0 \wedge -1 \wedge -2 \wedge \cdots.}\]
Note that $\FF$ is referred to as the \emph{spin module} in \cite{tkvdl}.

\begin{rmk}
We have used here the convention that Clifford algebra generators of different colours anti-commute. However, the Clifford algebra is sometimes defined by letting generators of different colours commute (see for example \cite[Section 1.2]{savlic}). In this case it is necessary to modify the action of $\Cl$ on $\FF$ to
\[\textstyle{ \psi_\ell(i)|_{\FF_{\cc}} \mapsto  1^{\otimes \ell-1} \otimes \psi(i) \otimes 1^{\otimes s-\ell}, \qquad \psi_\ell^*(i)|_{\FF_{\cc}} \mapsto 1^{\otimes \ell-1} \otimes \psi^*(i) \otimes 1^{\otimes s-\ell}.}\]
\end{rmk}

\begin{rmk}
As was the case with the Heisenberg algebra, the $s$-coloured Clifford algebra is isomorphic to the $1$-coloured Clifford algebra via relabelling of the indices. Moreover, by identifying the $s$-coloured and $1$-coloured Clifford algebras, $s$-coloured  fermionic Fock space is isomorphic to $1$-coloured fermionic Fock space. Thus, we will often refer to the Clifford algebra and fermionic Fock space without specifying the number of colours.
\end{rmk}

In the sequel, it will be useful to think of fermionic Fock space not only in terms of semi-infinite monomials, but also in terms of \emph{Young diagrams}. Recall that a Young diagram is a finite collection of boxes arranged in rows and columns such that the number of boxes in the $i$-th row is greater than or equal to the number of boxes in the $(i+1)$-st row. There are different conventions regarding the way to draw Young diagrams, but we will choose the English notation. That is, our rows will be left-justified and subsequent rows are placed underneath the previous one, as illustrated in Figure~\ref{diagram:young}.
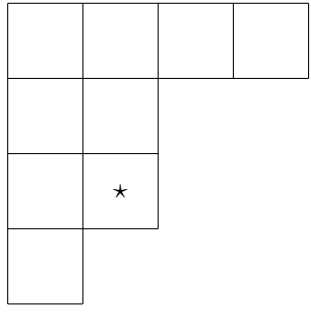
\begin{figure}
\begin{tikzpicture}

\draw (0,0) -- (0,-4); \draw (1,0) -- (1,-4); \draw (2,0) -- (2,-3); \draw (3,0) -- (3,-1); \draw (4,0) -- (4,-1);

\draw (0,0) -- (4,0); \draw (0,-1) -- (4,-1); \draw (0,-2) -- (2,-2); \draw (0,-3) -- (2,-3); \draw (0,-4) -- (1,-4);

\draw (1.5,-2.5) node {$\star$};

\end{tikzpicture}
\caption{Young diagram} \label{diagram:young}
\end{figure}
We say that a box is in the $(i,j)$-th \emph{position} if it is in the $i$-th row and $j$-th column of the diagram. For example, in Figure~\ref{diagram:young}, the box labelled with a ``$\star$" is in the $(3,2)$-th position. We may thus view a Young diagram $\lambda$ as being a subset of $(\N^+)^2$, where $(i,j) \in \lambda$ if and only if $\lambda$ has a box in the $(i,j)$-th position. For any Young diagram $\lambda$ and any $k \in \N^+$, let $\lambda_k$ denote the number of boxes in its $k$-th row ($\lambda_k = 0$ if there are no boxes in the $k$-th row). For any $(i,j) \in (\N^+)^2$, define the \emph{arm} and \emph{leg length} of $(i,j)$ to be
\[\textstyle{ a_\lambda(i,j) := \lambda_i - j, \quad \text{and} \quad \ell_\lambda(i,j) := \max\{ i' \in \N^+ \; | \; (i',j) \in \lambda \} - i,}\]
respectively. Intuitively, for each $(i,j) \in \lambda$, $a_\lambda(i,j)$ and $\ell_\lambda(i,j)$ count the number of boxes to the right of and below $(i,j)$, respectively. For two Young diagrams $\lambda$ and $\mu$, we define the \emph{relative hook length} of $(i,j)$ to be
\begin{equation}\textstyle{\label{eq:hook}
h_{\lambda,\mu}(i,j) := a_\lambda(i,j) + \ell_\mu(i,j) + 1.
}\end{equation}
If $\lambda = \mu$, we will simply write $h_\lambda(i,j) := h_{\lambda,\lambda}(i,j)$. Finally, we define the \emph{residue} of box in the $(i,j)$-th position to be $j-i$.

Clearly, every weakly decreasing sequence of non-negative integers
\[\textstyle{ (\lambda_1, \lambda_2, \lambda_3, \dots), }\]
such that $\lambda_k=0$ for $k \gg 0$ uniquely defines a Young diagram. Thus, one has a bijection between the set of semi-infinite monomials of a fixed charge and the set of Young diagrams. Indeed, suppose
\[\textstyle{ I = i_1 \wedge i_2 \wedge i_3 \wedge \cdots,}\]
is a semi-infinite monomial of charge $c$. Define $\lambda(I)$ to be the Young diagram determined by
\begin{equation}\textstyle{\label{eq:montoyoung}
\lambda(I)_k = i_k - c + k - 1.
}\end{equation}

\begin{lemma}\label{lem:montoyoung}
The mapping $I \mapsto \lambda(I)$ is a bijection from the set of semi-infinite monomials of charge $c$ to the set of Young diagrams. Moreover, if $I$ is a semi-infinite monomial of energy $n$, then $\lambda(I)$ consists of $n$ boxes.
\end{lemma}
Hence, by Lemma \ref{lem:montoyoung}, we may think of $F_c$ as the complex vector space with basis the set of all Young diagrams, thus giving us a Young diagram interpretation of fermionic Fock space.

It is well-known that the set of \emph{Schur functions} $\{s_\lambda\}$, where $\lambda$ runs over all Young diagrams, form a basis of $\Lambda$ (see statement (3.3) of \cite{mac}). We then have a vector space isomorphism
\begin{equation}\textstyle{ \label{eq:bfcorr}
B_c \to F_c, \quad s_\lambda \otimes q^c \mapsto \lambda,
}\end{equation}
for each $c \in \Z$. This induces an isomorphism $\BB \xrightarrow{\cong} \FF$, called the \emph{boson-fermion correspondence}.

We now recall a few facts about the representation theory of semisimple Lie algebras. Let $\g$ be a finite-dimensional semisimple Lie algebra and let $\widehat\g$ denote the corresponding untwisted affine Lie algebra of $\g$. Let $V$ be a representation of $\widehat\g$. One can show that for any Heisenberg subalgebra of $\widehat\g$ (under certain restrictions), the representation $V$ is isomorphic, as a Heisenberg algebra representation, to a direct sum of copies of the ring of symmetric functions $\Lambda$ (see \cite[Lemma 14.4(b)]{kac}). In particular, given a Heisenberg subalgebra of $\widehat\g$, one can give an explicit realization of $V_\text{basic}(\widehat\g)$ completely in terms of (a direct sum of copies of) $\Lambda$. Thus, to classify all such realizations, one needs to classify all inequivalent Heisenberg subalgebras of $\widehat\g$. For $\g$ semisimple, it is known that, up to conjugacy under the adjoint action of the associated Kac-Moody group, the Heisenberg subalgebras of $\widehat\g$ are parametrized by the conjugacy classes of the Weyl group of $\g$ (see \cite[Section 9]{kacpet}, in which the authors construct all possible Heisenberg subalgebras of $\widehat\g$). As mentioned in the introduction, the Heisenberg subalgebras of $\widehat\sla_r$ are thus parametrized by partitions of $r$. The methods of \cite[Section 9]{kacpet} can be extended to $\widehat\gl_r$, thus producing one Heisenberg subalgebra of $\widehat\gl_r$ for each partition of $r$.

The most well-known Heisenberg subalgebra of $\widehat\gl_r$ is the \emph{principal} Heisenberg subalgebra, which corresponds to the partition $\underline{r} = (r)$. Let $E_i, F_i, H_i$, for $i=0,1,\dots,r-1$ be the Chevalley generators of $\widehat\sla_r$. Then recall that the principal Heisenberg subalgebra of $\widehat\gl_r$ is given by
\[\textstyle{ \alpha(n) = \begin{cases} (E_0 + E_1 + \dots + E_{r-1})^n, & \text{for } n>0, \\ (F_0 + F_1 + \dots + F_{r-1})^{-n} & \text{for } n<0,\end{cases}}\]
By setting
\[\textstyle{ \alpha(0) = I \otimes 1,}\]
the $\alpha(n)$, $n \in \Z$, together with the central element $c \in \widehat\gl_r$, form a basis of a $1$-coloured oscillator algebra. To construct the Heisenberg subalgebra associated to an arbitrary partition $(r_1,\dots,r_s)$, we divide the matrices in $\gl_r$ into $s^2$ blocks, with the $(i,j)$-th block being of size $r_i \times r_j$ for all $i,j = 1,\dots,s$. The $(\ell,\ell)$-th diagonal block corresponds to the subalgebra $\gl_{r_\ell} \subseteq \gl_r$, and on the level of affine algebras, $\widehat\gl_{r_\ell} \subseteq \widehat\gl_r$. Let $E_k^\ell, F_k^\ell, H_k^\ell$ be the Chevalley generators for $\widehat\sla_{r_\ell} \subseteq \widehat\gl_{r_\ell}$. The Heisenberg subalgebra corresponding to the partition $(r_1,\dots,r_s)$, presented as an $s$-coloured algebra, is given by
\[\textstyle{ \alpha_\ell(n) = \begin{cases} (E^\ell_0 + E^\ell_1 + \dots + E^\ell_{r_\ell-1})^n, & \text{for } n>0, \\ (F_0^\ell + F_1^\ell + \dots + F_{r_\ell-1}^\ell)^{-n}, & \text{for } n<0, \end{cases} }\]
for $\ell=1,\dots,s$. It is worth noting that, restricted to the $\ell$-th colour, the $\alpha_\ell(n)$ determine the principal Heisenberg subalgebra of $\widehat\gl_{r_\ell}$. Set
\[\textstyle{ \alpha_\ell(0) = I_\ell \otimes 1,}\]
where $I_\ell$ is the identity matrix in $\gl_{r_\ell}$. Then the $\alpha_\ell(n)$ define an $s$-coloured oscillator algebra.

The various realizations of $V_\text{basic} = V_\text{basic}(\widehat\gl_r)$ were described explicitly in \cite{tkvdl} in terms of vertex operators on bosonic and fermionic Fock space. We briefly summarize the main theorem of that paper here. Fix a partition, $(r_1,\dots,r_s)$, of $r$. Let $R' = \lcm\{r_1,\dots,r_s\}$ and define
\[\textstyle{ R = \begin{cases} R', & \text{if } R' \left( \frac{1}{r_i} + \frac{1}{r_j}\right) \in 2 \Z \text{ for all } i,j, \\ 2R', & \text{if } R' \left( \frac{1}{r_i} + \frac{1}{r_j}\right) \notin 2 \Z \text{ for some } i,j. \end{cases} }\]
Define the \emph{normal ordering} on the $s$-coloured Clifford algebra generators
\[\textstyle{ :\psi_\ell(i) \psi^*_k(j): \; = \begin{cases} \psi_\ell(i) \psi^*_k(j), & \text{if } j>0, \\ -\psi^*_k(j) \psi_\ell(i), & \text{if } j \le 0. \end{cases}}\]
We introduce the formal $s$-coloured fermionic and bosonic fields:
\[\textstyle{ \psi_\ell(z) := \sum_{k \in \Z} \psi_\ell(i) z^{(R/r_\ell)k}, \quad \psi_\ell^*(z) := \sum_{k \in \Z} \psi_\ell^*(k) z^{-(R/r_\ell)k},}\]
\[\textstyle{ \alpha_\ell(z) = \sum_{k \in \Z} \alpha_\ell(k) z^{-(R/r_\ell)k} := \; :\psi_\ell(z) \psi^*_\ell(z):,}\]
This brings us to the main theorem of \cite{tkvdl} (as stated at the end of the introduction).

\begin{prop}\label{thm:tkvdl}
Let $\omega=e^{2\pi i/R}$. Then the homogeneous components of
\[\textstyle{ :\psi_\ell(\omega^p z) \psi^*_k(\omega^q z): - \delta_{k \ell} \frac{\omega^{(R/r_\ell)(p-q)}}{1-\omega^{(R/r_\ell)(p-q)}}, \quad {{(\ell \ne k, \; 1 \le p \le r_\ell, \; 1 \le q \le r_k)} \atop { \text{or } (\ell = k, \; 1 \le p \ne q \le r_\ell),}}}\]

\[\textstyle{ :\psi_\ell(z) \psi_\ell^*(z):, \quad 1 \le \ell \le s,}\]
together with the identity operator, span a Lie algebra of operators on $\FF$ ($s$-coloured fermionic Fock space) that is isomorphic to $\widehat\gl_r$. Moreover, via this identification with $\widehat\gl_r$,
\[\textstyle{ \FF(0) \cong V_\textup{basic},}\]
as $\widehat\gl_r$-modules and the $\alpha_\ell(k)$, $k \in \Z-\{0\}$ (resp.\ $k \in \Z$), together with the identity operator, give a basis of the Heisenberg subalgebra (resp.\ oscillator subalgebra) of $\widehat\gl_r$ associated to the partition $\underline{r}$.
\end{prop}

Note that the isomorphism $\FF(0) \xrightarrow{\cong} V_\text{basic}$ in Proposition \ref{thm:tkvdl} is determined by $|0\rangle^{\otimes s} \mapsto v_0$.

The goal of this paper is to give a geometric version of the various realizations of $V_\text{basic}$. Our strategy will be as follows. We first construct geometric analogues the oscillator and Clifford algebra representations on bosonic and fermionic Fock space. Our method for doing this will be similar to \cite[Section 3]{savlic}. Next, we would like to construct an action of $\widehat\gl_r$ on our geometric fermionic Fock space in the spirit of Proposition \ref{thm:tkvdl}. To do this, we recall that $\widehat\gl_r$ is generated by the Chevalley generators $E_k, F_k, H_k$, $k=0,1,\dots,r-1$ of $\widehat\sla_r$ and loops on the identity $I \otimes t^n$, $n \in \Z$. Thus, to define an action of $\widehat\gl_r$ on our geometric Fock space, it is enough to define the action of $E_k, F_k, H_k$ and $I \otimes t^n$. Algebraically, the action of $\widehat\gl_r$ on fermionic Fock space is given by the homogeneous components of vertex operators as in Proposition \ref{thm:tkvdl}. We explicitly describe the algebraic action of $E_k, F_k$ and $I \otimes t^n$ in Lemma \ref{lem:glaction} below (the action of $H_k$ is determined by $E_k$ and $F_k$ since $H_k = [E_k,F_k]$). First, we observe that for each $k = 0,1,\dots, r-1$, we can write
\[\textstyle{ k = r_1 + \dots + r_{\ell-1} + k',}\]
for unique $1 \le \ell \le s$ and $0 \le k' \le r_{\ell-1}$.

\begin{lemma}\label{lem:glaction}
The action of $\widehat\gl_r$ on $\FF$ from Proposition \ref{thm:tkvdl} is given by the map $\rho:\widehat\gl_r \to \gl(\FF$), described below. For each $k = 0,1,\dots,r-1$, write $k = r_1 + \dots + r_{\ell-1} + k'$. If $k' \ne 0$, then
\begin{equation}\textstyle{\label{eq:E1}
\rho(E_k) = \sum_{i \in \Z} \psi_\ell(k' + i r_\ell) \psi_\ell^*(k' + i r_\ell + 1),
}\end{equation} \vspace{-10pt}
\begin{equation}\textstyle{
\rho(F_k) = \sum_{i \in \Z} \psi_\ell(k' + i r_\ell + 1) \psi_\ell^*(k' + i r_\ell).
}\end{equation}
If $k'=0$ and $\ell \ne 1$,
\begin{equation}\textstyle{
\rho(E_k) = \sum_{i \in \Z} \psi_{\ell-1}((i+1)r_{\ell-1}) \psi_{\ell}^*(i r_{\ell} + 1),
}\end{equation}\vspace{-10pt}
\begin{equation}\textstyle{
\rho(F_k) = \sum_{i \in \Z} \psi_{\ell} (ir_{\ell} + 1) \psi_{\ell-1}^*((i+1) r_{\ell-1}).
}\end{equation}
If $k'=0$ and $\ell=1$, (i.e.\ $k=0$)
\begin{equation}\textstyle{
\rho(E_0) = \sum_{i \in \Z} \psi_s(i r_s) \psi_1^*(i r_1 + 1),
}\end{equation}\vspace{-10pt}
\begin{equation}\textstyle{
\rho(F_0) = \sum_{i \in \Z} \psi_1(i r_1 + 1) \psi_s^*(i r_s).
}\end{equation}
Finally,
\begin{equation}\textstyle{
\rho(I \otimes t^n) = \sum_{\ell = 1}^s \alpha_\ell(n r_\ell).
}\end{equation}
\end{lemma}

\begin{proof}
The proof consists of picking out the appropriate components of the vertex operators in Proposition \ref{thm:tkvdl}, thus we leave the details to the reader. For an explicit proof, see \cite[Lemma 1.21]{me}.
\end{proof}

%
%

\section{Quiver Varieties}\label{section:2}

In this section, we review some of the geometric concepts we will need to construct our geometric realizations of $V_\text{basic}(\widehat\gl_r)$. Our main geometric objects of interest will be \emph{Nakajima quiver varieties}; these are examples of \emph{geometric (invariant) quotients}. We will briefly recall the necessary concepts of geometric invariant theory as they are needed. For a more in-depth treatment of geometric invariant theory see \cite{mum}. For more information on Nakajima quiver varieties see \cite{nak2} and \cite{nak1}.

Throughout this section, and all subsequent sections, for a group $G$ acting on a set $X$, we denote by $X^G$ the set of $G$-fixed points of $X$.

Let $Q=(Q_0,Q_1)$ be a quiver. For all $\rho \in Q_1$, write $t(\rho)$ and $h(\rho)$ for the tail and the head of $\rho$, respectively.  Let $\widetilde{Q}=(\widetilde{Q}_0,\widetilde{Q}_1)$ be the double quiver of $Q$. That is,
\[\textstyle{ \widetilde{Q}_0 = Q_0, \qquad \text{and} \qquad \widetilde{Q}_1 = Q_1 \cup \overline{Q}_1, }\]
where $\overline{Q}_1$ is the set of arrows in $Q_1$ with orientation reversed. We then have a natural involution $\bar{ }: \QQ_1 \rightarrow \QQ_1$, which sends every arrow in $Q_1$ to its corresponding reverse arrow in $\overline{Q}_1$ and vice versa. We define the function $\varepsilon: \QQ_1 \rightarrow \{\pm1\}$ by
\[\textstyle{ \varepsilon(\rho) = \begin{cases} 1, & \text{if } \rho \in Q_1, \\ -1, & \text{if } \rho \in \overline{Q}_1. \end{cases}}\]
Let $V = \bigoplus_{k \in \widetilde{Q}_0} V_k$ and $W=\bigoplus_{k \in \QQ_0} W_k$ be finite-dimensional $\widetilde{Q}_0$-graded complex vector spaces. Define
\[\textstyle{ E_V = \bigoplus_{\rho \in \QQ_1} \Hom(V_{t(\rho)}, V_{h(\rho)}), \qquad L_{V,W} = \bigoplus_{k \in \QQ_0} \Hom(V_k, W_k). }\]
We define a ``multiplication" $E_V \times E_V \rightarrow L_{V,V}$ given by
\[\textstyle{ (AB)_k = \sum_{\rho \in \QQ_1, t(\rho)=k} A_{\bar{\rho}} B_\rho, }\]
for all $A=(A_\rho), B=(B_\rho) \in E_V$. We then define
\[\textstyle{ \mathbf{M} = \mathbf{M}(V,W) = E_V \oplus L_{W,V} \oplus L_{V,W}. }\]
The function $\varepsilon:\QQ_1 \rightarrow \{\pm1\}$ induces a function $\varepsilon:E_V \rightarrow E_V$ given by
\[\textstyle{ \varepsilon(C)_\rho = \varepsilon(\rho) C_\rho. }\]
We have a symplectic form on $\omega$ on $\mathbf{M}$ given by
\[\textstyle{ \omega((C_1,i_1,j_1), (C_2,i_2,j_2)) = \tr(\varepsilon(C_1)C_2) + \tr(i_1 j_2 - i_2 j_1). }\]
Let $G_V = \prod_{k \in \QQ_0} \GL(V_k)$. Then $G_V$ acts on $\mathbf{M}$ via
\[\textstyle{ g \cdot (C,i,j) = (gCg^{-1}, gi, jg^{-1}), }\]
for all $g \in G_V$ and $(C,i,j) \in \bold{M}$. Then $G_V$ is an algebraic group whose action on $\mathbf{M}$ preserves the symplectic form $\omega$. The moment map vanishing at the origin is given is
\[\textstyle{ \mu: \mathbf{M} \to L_{V,V}, \quad (C,i,j) \mapsto \varepsilon(C)C + ij, }\]
where the Lie algebra $L_{V,V}$ of $G_V$ is identified with its dual via the trace.

\begin{rmk}
Note that the set $\mathbf{M}$ and the map $\mu$ do not depend on the orientation of the quiver $Q$. Thus, our construction above, as well as the definition of quiver varieties below, depend only on the underlying (undirected) graph of $Q$.
\end{rmk}

\begin{defn}[Invariant, Stable]\label{defn:stable} Let $S=\bigoplus_{k \in \QQ_0} S_k$, where each $S_k$ is a subspace of $V_k$. For $C \in E_V$, we say that $S$ is $C$\emph{-invariant} if $C_\rho(S_{t(\rho)}) \subseteq S_{h(\rho)}$ for all $\rho \in \QQ_1$.

An element $(C,i,j) \in \mathbf{M}$ is called \emph{stable} if the following condition holds: if $S$ is a $\QQ_0$-graded subspace of $V$ such that $S$ is $C$-invariant and $\im(i) \subseteq S$, then $S=V$. We denote by $\mathbf{M}^{\st}$ the set of stable points of $\mathbf{M}$.
\end{defn}

Let $G$ be a reductive group acting freely on an algebraic variety $X$. Then for every $G$-invariant open set $U \subseteq X$, $G$ acts on $\OO_X(U)$ by
\[\textstyle{ (g \cdot f)(x) = f(g^{-1} \cdot x),}\]
for all $g \in G$, $f \in \OO_X(U)$ and $x \in U$. A function $f \in \OO_X(U)^G$ then determines a well-defined function $\overline{f}: U/G \to \C$ given by
\[\textstyle{ \overline{f}( G \cdot x) = f(x).}\]
The \emph{geometric quotient} (of $X$ by $G$), is the topological quotient $\pi: X \twoheadrightarrow X/G$ equipped with the structure sheaf
\[\textstyle{ \OO_{X/G}(U) = \left\{ \overline{f}: U \to \C \; | \; f \in \OO_X(\pi^{-1}(U))^G \right\}.}\]
Note that this definition of a geometric quotient is a slightly simplified version of \cite[Definition 0.6]{mum}. One can show that $\mathbf{M}^{\st}_0$ is a quasi-affine algebraic variety and that $G_V$ acts freely on $\mathbf{M}^{\st}_0$. 

\begin{defn}[Nakajima Quiver Variety]\label{defn:quivervariety}
The \emph{Nakajima quiver variety} (associated to $Q$) is the geometric quotient
\[\textstyle{ \MM = \MM(V,W) := \mathbf{M}^{\st}_0(V,W) / G_V. }\]
We will write $[C,i,j]_{G_V}$ (or simply $[C,i,j]$ when there is no risk of confusion) to denote the $G_V$-orbit of the point $(C,i,j) \in \mathbf{M}_0^{\st}$.
\end{defn}

Note that Definition \ref{defn:quivervariety} is slightly different than the definition of quiver varieties found in \cite[Section 3.ii]{nak2}, however the two definitions are merely duals of each other. Indeed, let $V^*$ and $W^*$ be the duals of $V$ and $W$, respectively. Then an element $(C,i,j) \in \mathbf{M}(V,W)$ naturally induces an element $(C^*,j^*,i^*) \in \mathbf{M}(V^*,W^*)$ by taking the transposes of the appropriate maps. One can then show that the mapping $(C,i,j) \mapsto (C^*,j^*,i^*)$ induces an isomorphism from $\MM$ to the quiver variety defined in \cite{nak2} (see specifically \cite[Corollary 3.12]{nak2}). 

By \cite[Lemma 3.10, Corollary 3.12]{nak2}, the varieties $\mathbf{M}_0^{\st}$ and $\MM$ are smooth. The tangent space of $\mathbf{M}_0^{\st}$ at a point $(C,i,j)$ is
\[\textstyle{ \T_{(C,i,j)}(\mathbf{M}_0^{\st}) = \T_{(C,i,j)}(\mu^{-1}(0)) = \ker \dd \mu,}\]
where $\dd \mu$ is the differential of the moment map $\mu$ at $(C,i,j)$, and is given by
\[\textstyle{ \dd \mu: \mathbf{M} \rightarrow L_{V,V}, \quad (D,a,b) \mapsto \varepsilon(C)D+\varepsilon(D)C+ib+aj. }\]
Let $\sigma = \sigma^{(C,i,j)}: G_V \to \mathbf{M}_0^{\st}$ be the map given by $g \mapsto g \cdot (C,i,j)$. The differential $\dd \sigma: L_{V,V} \to \ker \dd \mu$ is then given by
\[\textstyle{ \dd \sigma(X) = (XC-CX,Xi,-jX).}\]

\begin{lemma}\label{lem:tanspaceofqv}
Let $(C,i,j) \in \mathbf{M}_0^{\st}$. Then
\begin{enumerate}
	\item $\dd \sigma$ is injective and $\dd \mu$ is surjective, and
	
	\item the tangent space of $\MM$ at $[C,i,j]$ may be identified with the middle cohomology of the following complex:
\begin{equation}\textstyle{\label{eq:complex}
L_{V,V} \xhookrightarrow{\dd\sigma} E_V \oplus L_{W,V} \oplus L_{V,W} \xtwoheadrightarrow{\dd \mu} L_{V,V}.
}\end{equation}
\end{enumerate}
\end{lemma}

\begin{proof}
The proof of (1) is completely analogous to \cite[Lemma 3.2]{savlic}. Part (2), is the dual of \cite[Corollary 3.12]{nak2}.
\end{proof}

It is worth noting that, up to isomorphism, the varieties $\mathbf{M}_0^{\st}(V,W)$ and $\MM(V,W)$ are parametrized by the graded dimensions of $V$ and $W$. Indeed, if $V'$ and $W'$ are vector spaces of the same graded dimension as $V$ and $W$, respectively, then there exist graded vector space isomorphisms $V \to V'$ and $W \to W'$, which in turn induce isomorphisms of varieties $\mathbf{M}_0^{\st}(V,W) \to \mathbf{M}_0^{\st}(V',W')$ and $\MM(V,W) \to \MM(V',W')$. Hence, we define one standard representative for each isomorphism class: for $\vv = (\vv_k), \ww = (\ww_k) \in (\N)^{\QQ_0}$, let
\[\textstyle{ \mathbf{M}_0^{\st}(\vv,\ww) := \mathbf{M}_0^{\st} \left( \bigoplus_{k \in \QQ_0} \C^{\vv_k}, \bigoplus_{k \in \QQ_0} \C^{\ww_k}\right),}\]
and
\[\textstyle{\MM(\vv,\ww) := \MM \left( \bigoplus_{k \in \QQ_0} \C^{\vv_k}, \bigoplus_{k \in \QQ_0} \C^{\ww_k}\right) = \mathbf{M}_0^{\st}(\vv,\ww) / G_\vv, }\]
where $G_\vv = \prod_{k \in \QQ_0} \GL_{\vv_k}(\C)$.

From now on, we restrict ourselves to the case where $Q$ is a quiver of type $\widehat{A}_{m-1}$, $m \in \N^+$ (with the case $m=1$ corresponding to the quiver consisting of one vertex and one loop). We label the vertices of $Q$ by $\{0,1, \dots, m-1\}$, and hence we can identify $Q_0=\QQ_0$ with $\Z_m$. In this case, we have that $\QQ$ is the quiver:

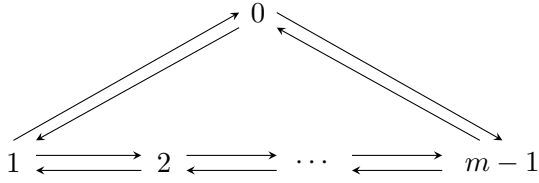
\begin{center}
\begin{tikzpicture}[>=stealth]

\draw (0,0) node {$1$} (2,0) node {$2$} (4,0) node {$\cdots$} (6.5,0) node {$m-1$} (3.25,2) node {$0$};

\draw[->] (0.3,.1) -- (1.7,.1); \draw[->] (1.7,-.1) -- (0.3,-.1);
\draw[->] (2.3,.1) -- (3.5,.1); \draw[->] (3.5,-.1) -- (2.3,-.1);
\draw[->] (4.5,.1) -- (5.7,.1); \draw[->] (5.7,-.1) -- (4.5,-.1);

\draw[->] (0,0.3) -- (3,2); \draw[->] (3,1.8) -- (0.3,0.3);
\draw[->] (3.5,2) -- (6.5,0.3); \draw[->] (6.2,0.3) -- (3.5,1.8);

\end{tikzpicture}
\end{center}
We denote by $\mathbf{M}(m; \vv,\ww)$ (resp.\ $\MM(m;\vv,\ww)$) the variety $\mathbf{M}(\vv,\ww)$ (resp.\ $\MM(\vv,\ww)$) corresponding to a quiver of type $\widehat{A}_{m-1}$ (recall that these varieties do not depend on the orientation of $Q$, thus $\mathbf{M}(m; \vv,\ww)$ and $\MM(m;\vv,\ww)$ are well-defined). Of particular interest to us will be the case $m=1$, for which we introduce the special notation:
\[\textstyle{ M(s,n) := \mathbf{M}(1;(n),(s)), \quad \M(s,n) := \MM(1;(n),(s)),}\]
where $s,n \in \N$. Note that in this case
\[\textstyle{ M(s,n) = \Hom(\C^n,\C^n) \oplus \Hom(\C^n,\C^n) \oplus \Hom(\C^s,\C^n) \oplus \Hom(\C^n,\C^s),}\]
and so we will denote elements of $M(s,n)$ by $(A,B,i,j)$ and the elements of $\M(s,n)$ by $[A,B,i,j]$, where, by convention, $A$ represents the linear map associated to the loop in $Q_1$ and $B$ the linear map associated to the loop in $\overline{Q}_1$. The moment map $\mu$ then simplifies to
\[\textstyle{ \mu(A,B,i,j) = [A,B] + ij.}\]
Thus, by \cite[Theorem 2.1]{nak1}, $\M(s,n)$ is isomorphic to the moduli space of framed torsion-free sheaves on $\mathbb{P}^2$ with rank $s$ and second Chern class $c_2 = n$.

Fix an $(s+1)$-dimensional torus
\[\textstyle{ T = (\C^*)^s \times \C^*, }\]
and denote elements of $T$ by $(e,t)$, where $e=(e_1,\dots,e_s) \in (\C^*)^s$ and $t \in \C^*$. We have a natural action of $(\C^*)^s$ on $\C^s$ given by
\[\textstyle{ e \cdot (w_1, \dots, w_s) = (e_1 w_1, \dots, e_s w_s),}\]
for all $e \in (\C^*)^s$ and $(w_1, \dots, w_s) \in \C^s$. For each $\cc \in \Z^s$, we have an action of $T$ on $M(s,n)$ given by
\[\textstyle{ (e,t) \cdot (A,B,i,j) = (tA, t^{-1}B, i e^{-1} t^{-\cc}, e t^{\cc} j),}\]
where
\[\textstyle{ t^{\cc} := (t^{\cc_1},\dots,t^{\cc_s}).}\]
The torus action preserves the space $M_0^{\st}(s,n)$ and commutes with the action of $\GL_n(\C)$, thus we have a well-defined action on $\M(s,n)$:
\begin{equation}\textstyle{\label{eq:torus}
(e,t) \cdot [A,B,i,j] = [tA, t^{-1}B, ie^{-1}t^{-\cc},et^{\cc}j].
}\end{equation}
Let $\M_{\cc}(s,n)$ denote the moduli space $\M(s,n)$ with the torus action given by Equation \eqref{eq:torus}.

Consider the case where $s=1$. It is known that if $(A,B,i,j) \in M_0^{\st}(1,n)$ then $j=0$ (see \cite[Proposition 2.7]{nak1}), thus we will simply write $(A,B,i)$ for an element $(A,B,i,0) \in M_0^{\st}(1,n)$. Let $\omega := e^{2 \pi \sqrt{-1} / m}$. The finite cyclic group of order $m$ acts on $\M_c(1,n)$ via the embedding
\begin{equation}\textstyle{ \label{eq:Zmaction}
\Z_m \hookrightarrow T, \quad k \mapsto (1,\omega^k).
}\end{equation}
That is
\[\textstyle{ k \cdot [A,B,i] = [\omega^k A, \omega^{-k}B, \omega^{-k c}i], }\]
for all $k \in \Z_m$ and $[A,B,i] \in \M_c(1,n)$. Recall that the fixed point set $\M_c(1,n)^{\Z_m}$ is a closed subvariety of $\M_c(1,n)$. The goal for the remainder of this section will be to describe $\M_c(1,n)^{\Z_m}$ in terms of quiver varieties of type $\widehat{A}_{m-1}$.

\begin{rmk}\label{rmk:F}
Let
\[\textstyle{ F_c(n) :=\left\{ (A,B,i) \in M_0^{\st}(1,n) \; | \; [A,B,i] \in \M_c(1,n)^{\Z_m} \right\},}\]
i.e.\ $F_c(n)$ is the preimage of $\M_c(1,n)^{\Z_m}$ under the projection $M_0^{\st}(1,n) \twoheadrightarrow \M_c(1,n)$. Thus, $F_c(n)$ is a closed subvariety of $M_0^{\st}(1,n)$. We have that $(A,B,i) \in F_c(n)$ if and only if there exists a $g \in \GL_n(\C)$ such that
\begin{equation}\textstyle{ \label{eq:weightsp}
\omega A = g^{-1}A g, \quad \omega^{-1} B = g^{-1}Bg, \quad \omega^{-c} i = g^{-1}i.
}\end{equation}
Note that since $\GL_n(\C)$ acts freely on $M^{\st}_0$, such a $g$, if it exists, is unique. We then have a weight space decomposition
\[\textstyle{ \C^n = \bigoplus_{k \in \Z_m} V_k(A,B,i), \quad V_k(A,B,i) := \{ v \in \C^n \; | \; gv = \omega^k v \}.}\]
By the stability of $(A,B,i)$, we have that $\C^n$ is spanned by elements of the form $A^p B^q i(1)$, for $p, q \in \N$. One can then verify that
\begin{equation}\textstyle{\label{eq:Vk}
V_k(A,B,i) = \vspan\{ A^p B^q i \; | \; (q-p) \equiv (k - c) \mod m \}.
}\end{equation}
We observe that $\im i \subseteq V_{\bar{c}}$ (where $\bar{c}$ is the equivalency class of $c$ in $\Z_m$) and the restrictions of $A$ and $B$ to $V_k$ yield maps
\[\textstyle{ A|_{V_k(A,B,i)}: V_k(A,B,i) \rightarrow V_{k-1}(A,B,i) \quad \text{and} \quad B|_{V_k(A,B,i)}: V_k(A,B,i) \rightarrow V_{k+1}(A,B,i). }\]
Conversely, suppose we have a weight space decomposition $\C^n=\bigoplus_{k \in \Z_m} V_k$ such that $A|_{V_k}:V_k \rightarrow V_{k-1}$ and $B|_{V_k}:V_k \rightarrow V_{k+1}$ and $\im i \subseteq V_{\bar{c}}$. It is easy to see that
\[\textstyle{ g = \prod_{k \in \Z_m} \omega^k \id_{V_k}.}\]
satisfies conditions \eqref{eq:weightsp} and that $V_k(A,B,i) = V_k$ for each $k \in \Z_m$.
\end{rmk}

\begin{lemma}\label{lem:fixedpointsmooth}
The variety $\M_c(1,n)^{\Z_m}$ is smooth. Moreover, the tangent space of $\M_c(1,n)^{\Z_m}$ at $[A,B,i]$ may be identified with the middle cohomology of the following complex:
\begin{equation}\textstyle{\label{eq:complex2}
L \xhookrightarrow{\dd\sigma} E^- \oplus E^+ \oplus \Hom(\C,V_{\bar{c}}) \oplus \Hom(V_{\bar{c}},\C) \xtwoheadrightarrow{\dd \mu} L,
}\end{equation}
where $L = \bigoplus_{k \in \Z_m} \Hom(V_k,V_k)$ and $E^{\pm} = \bigoplus_{k \in \Z_m} \Hom(V_k, V_{k \pm 1})$, and $\sigma:\GL_n(\C) \to M_0^{\st}$ is the map given by $g \mapsto g \cdot (A,B,i)$.
\end{lemma}

\begin{proof}
For the first part, note that the category of finite-dimensional linear representations of $\Z_m$ is semisimple. Therefore, by \cite[Proposition 1.3]{iversen}, $\M_c(1,n)^{\Z_m}$ is smooth. The second part can be proved by computing the $\Z_m$-fixed points of Complex \eqref{eq:complex}.
\end{proof}

Let $(C,i,j) \in \mathbf{M}(m; \vv,\1_{\bar{c}})$. We identify $\bigoplus_{k \in \Z_m} \C^{\vv_k}$ with $\C^{|\vv|}$ by identifying $\1^k_\ell$ with $\1_{\vv_{0} + \dots + \vv_{k-1} + \ell}$, where $\{\1^k_\ell\}_{\ell=1}^{\vv_k}$ and $\{\1_\ell\}_{\ell=1}^{|\vv|}$ denote the standard bases of $\C^{\vv_k}$ and $\C^{|\vv|}$, respectively. Let $A_C, B_C \in \Hom\left(\C^{|\vv|},\C^{|\vv|}\right)$ be the maps determined by
\begin{equation}\textstyle{\label{eq:A_C,B_C}
(A_C)|_{\C^{\vv_k}}= \varepsilon(\rho) C_\rho \qquad \text{and} \qquad  (B_C)|_{\C^{\vv_k}}=C_{\tau},
}\end{equation}
for each $k \in \Z_m$, $\rho:k \to k-1$ in $\QQ_1$ and $\tau:k\rightarrow k+1$ in $\QQ_1$ (note that in the $m=2$ case, we simply make a choice as to which arrow of $Q_1$ corresponds to $k \to k+1$ and which one corresponds to $k \to k-1$). We thus have a mapping
\begin{equation}\textstyle{ \label{eq:map}
\mathbf{M}(m; \vv, \1_{\bar{c}}) \to M(1,|\vv|), \quad (C,i,j) \mapsto (A_C, B_C, i, j).
}\end{equation}
Clearly, the map \eqref{eq:map} preserves stability and $\mu(C,i,j) = 0$ implies $\mu(A_C,B_C,i,j)=0$. In particular, if $(C,i,j) \in \mathbf{M}_0^{\st}(m;\vv,\1_{\bar{c}})$, then $j=0$. Moreover, by Remark \ref{rmk:F}, for each $(C,i,0) \in \mathbf{M}_0^{\st}(m;\vv,\1_{\bar{c}})$, we have that $(A_C,B_C,i) \in F_c(n)$ (with $V_k(A_C,B_C,i) = \C^{\vv_k}$). Thus, we have morphism of varieties:
\[\textstyle{ \varphi_{\vv}: \mathbf{M}_0^{\st}(m; \vv, \1_{\bar{c}}) \rightarrow F_c(|\vv|), \quad (C,i,0) \mapsto (A_C, B_C, i). }\] 

\begin{lemma}\label{lem:dir2}
The map $\varphi_{\vv}$ induces a morphism of varieties
\[\textstyle{ \overline{\varphi}_{\vv}: \MM(m;\vv,\1_{\bar{c}}) \rightarrow \M_c(1,|\vv|)^{\Z_m}, \quad [C,i,0] \mapsto [A_C,B_C,i]. }\]
In particular, one has a morphism of varieties $\overline{\varphi}: \coprod_{|\vv|=n} \MM(m; \vv, \1_{\bar{c}}) \to \M_c(1,n)^{\Z_m}$, given by $\overline{\varphi} = \coprod_{|\vv|=n} \overline{\varphi}_{\vv}$.
\end{lemma}

\begin{proof}
Let $\psi: G_{\vv} \to \GL_{|\vv|}(\C)$ be the group homomorphism induced by our identification of $\bigoplus_{k \in \Z_m} \C^{\vv_k}$ with $\C^{|\vv|}$, i.e.\ partitioning $\GL_{|\vv|}(\C)$ into $m^2$ blocks of size $\vv_i \times \vv_j$, $i,j=0,1,\dots,m-1$, we embed $\GL_{\vv_k}(\C)$ into the $(k,k)$-th diagonal block of $\GL_{|\vv|}(\C)$, for each $k = 0,1,\dots,m-1$. One then has the following commutative diagram:
\begin{center}
\begin{tikzpicture}[>=stealth]

\draw (0,0) node {$G_{\vv} \times \mathbf{M}_0^{\st}(m; \vv, \1_{\bar{c}})$} (4,0) node {$\mathbf{M}_0^{\st}(m; \vv, \1_{\bar{c}})$} (0,-1.5) node {$\GL_{|\vv|}(\C) \times F_c(|\vv|)$} (4,-1.5) node {$F_c(|\vv|),$};

\draw[->] (0,-0.3) -- (0,-1.2); \draw[->] (1.8,0) -- (2.7,0); \draw[->] (4,-0.3) -- (4,-1.2); \draw[->] (1.6,-1.5) -- (3.3,-1.5);

\draw (-0.7,-0.75) node {$\psi \times \varphi_{\vv}$} (4.4, -0.75) node {$\varphi_{\vv}$};

\end{tikzpicture}
\end{center}
where the horizontal arrows represent the group action. One therefore has a well-defined map $\overline{\varphi}_{\vv}: \MM(m;\vv,\1_{\bar{c}}) \to \M_{\cc}(1,|\vv|)^{\Z_m}$, on the quotients.
\end{proof}

\begin{thm}\label{thm:variso}
The map $\overline{\varphi}$ from Lemma \ref{lem:dir2} is an isomorphism. Therefore,
\[\textstyle{ \M_c(1,n)^{\Z_m} \cong \coprod_{|\vv|=n} \MM(m; \vv, \1_{\bar{c}}),}\]
as varieties.
\end{thm}

\begin{proof}
We first begin by showing that $\overline{\varphi}$ is bijective. Suppose that $(C,i,0) \in \mathbf{M}_0^{\st}(m; \vv, \1_{\bar{c}})$ and $(D,a,0) \in \mathbf{M}_0^{\st}(m; \uu, \1_{\bar{c}})$, with $|\vv|=|\uu|=n$, are such that
\[\textstyle{ \overline{\varphi}[C,i,0] = \overline{\varphi}[D,a,0],}\]
i.e.\ $[A_C, B_C, i] = [A_D, B_D, a]$. Then there exists $g \in \GL_n(\C)$ such that $(A_C, B_C, i) = g \cdot (A_D, B_D, a)$. For each $k \in \Z_m$,
\[\textstyle{ \C^{\vv_k} = V_k(A_C,B_C,i) = g (V_k(A_D,B_D,a)) = g(\C^{\uu_k}).}\]
In particular, $\vv = \uu$. Moreover, since $g|_{\C^{\vv_k}}: \C^{\vv_k} \to \C^{\vv_k}$, we may view $g$ as an element of $G_{\vv}$. One then easily verifies that $(C,i,0) = g \cdot (D,a,0)$. Thus, $[C,i,0] = [D,a,0]$, and so $\overline{\varphi}$ is injective. 

Now, let $(A,B,i) \in F_c(n)$. Set $\vv_k = \dim V_k(A,B,i)$ and choose a graded linear isomorphism
\[\textstyle{ f: \bigoplus_{k \in \Z_m} V_k(A,B,i) \to \bigoplus_{k \in \Z_m} \C^{\vv_k}.}\]
Then $[A,B,i] = [ f A f^{-1}, f B f^{-1}, f i]$. Define $(C,a,0) \in \mathbf{M} (m;\vv,\1_{\bar{c}})$ by
\[\textstyle{ C_\rho = \begin{cases} \varepsilon(\rho) (f A f^{-1})|_{\C^{\vv_k}}, & \text{if } \rho: k \to k-1, \\ (f B f^{-1})|_{\C^{\vv_k}}, & \text{if } \rho: k \to k+1,  \end{cases} \quad \text{and} \quad a = f i, }\]
for all $\rho \in \QQ_1$. It is easy to see that $(C,i,0)$ is stable and $\mu(C,a,0)=0$. Thus, $(C,a,0) \in \mathbf{M}_0^{\st}$, and, by construction, $\overline{\varphi}[C,i,0] = [A,B,i]$. Hence, $\overline{\varphi}$ is surjective.

Next, we show that $\overline{\varphi}$ induces an isomorphism on the tangent spaces. Recall that we identify the tangent spaces of $\MM$ and $\M^{\Z_m}$ with the middle cohomologies of Complex \eqref{eq:complex} and Complex \eqref{eq:complex2}, respectively. By construction of $\overline{\varphi}$, the induced map $\dd \overline{\varphi}$ on the tangent spaces is
\[\textstyle{ \dd \overline{\varphi}: \T_{[C,i,0]}(\MM) \to \T_{[A_C,B_C,i]}(\M^{\Z_m}), \quad
(D,a,b) + \im \dd \sigma \mapsto (D^-, D^+, a, b) + \im \dd \sigma, }\]
where $(D^-)_k = \varepsilon(\rho) D_{\rho:k \to k-1}$ and $(D^+)_k = D_{\rho:k \to k+1}$ for all $k \in \Z_m$. Clearly, $\dd \overline{\varphi}$ is injective. Moreover, by Lemma \ref{lem:tanspaceofqv} and Lemma \ref{lem:fixedpointsmooth}, $\dim \T_{[C,i,0]}(\MM) = \dim \T_{[A_C,B_C,i]}(\M^{\Z_m})$, and hence $\dd \overline{\varphi}$ is an isomorphism. Thus, $\overline{\varphi}$ is an isomorphism of varieties.
\end{proof}

%
%

\section{Vector Bundles and Geometric Operators}\label{section:3}

In this section, we describe how to obtain so-called ``geometric operators" on the localized equivariant cohomology of smooth algebraic varieties; this method was first introduced in \cite{ceo}. The theory outlined in this section will serve as our main tool for constructing our geometric versions of the Clifford, Heisenberg and Chevalley operators in the next section. We do not review equivariant cohomology theory here, but instead refer the reader to such expository papers as \cite{tym} or \cite{brion}. Our constructions rely heavily on the Localization Theorem (see Theorem \ref{thm:localization}). The reader may wish to consult \cite[Appendix to Chapter 6]{audin} for more information on localization.

Let $G = (\C^*)^d$ be a $d$-dimensional torus and, for each $j=1,\dots,d$, denote the $1$-dimensional $G$-module
\[\textstyle{ (g_1,\dots,g_d) \mapsto g_j,}\]
by $g_j$. Let $\pt$ denote the space consisting of a single point equipped with the trivial action of the torus $G$. Let $t_j$ denote the first Chern class of the vector bundle
\[\textstyle{ g_j \rightarrow \pt, }\]
for each $j=1,\dots,d$. Note that the $t_j$ are elements of degree $2$. Recall that the \emph{equivariant cohomology} of $\pt$ is
\[\textstyle{ H_G^*(\pt) = \C[t_1, \dots, t_d].}\]
Let $X$ be a topological space equipped with a $G$-action. Then $H_G^*(X)$ is a $H_G^*(\pt)$-module. We consider the \emph{localized equivariant cohomology} of $X$:
\[\textstyle{ \HH_G^*(X) := H_G^*(X) \otimes_{\C[t_1,\dots,t_d]} \C(t_1,\dots,t_d).}\]
Unless otherwise noted, ``cohomology" will always mean ``localized equivariant cohomology". Let
\[\textstyle{ i: X^G \hookrightarrow X, }\]
be the inclusion of the $G$-fixed points and let
\[\textstyle{ p: X^G \twoheadrightarrow \pt.}\]
The advantage of localized equivariant cohomology over nonlocalized equivariant cohomology is that its study can be reduced to the cohomology of the $G$-fixed points. We will only be interested in the case where $X$ is a smooth variety with finitely many $G$-fixed points. In this situation, we have the following theorem.

\begin{thm}[Localization Theorem]\label{thm:localization}
The following map is an isomorphism of algebras:
\begin{equation}\textstyle{ \label{eq:localization}
\HH^*_G(X) \to \HH^*_G(X^G)=\bigoplus_{x \in X^G} \HH^*_G(\pt), \quad \alpha \mapsto \left(\frac{i_x^*(\alpha)}{e_G(\T_x)}\right)_{x \in X^G},
}\end{equation}
where $i_x:\{x\}\hookrightarrow X$, $\T_x$ is the tangent space of $x$ in $X$, and $e_G(\T_x)$ is the equivariant Euler class of $\T_x$. The inverse of \eqref{eq:localization} is given by the $Gysin$ map $i_*: \HH_G^*(X^G) \rightarrow \HH^*_G(X)$.
\end{thm}

\begin{proof}
This is a restatement of \cite[Proposition 9.1.2]{cdks} in the case that $X$ has finitely many fixed points.
\end{proof}

Suppose now that $X$ has real dimension $4n$, for some $n \in \N$. We define a bilinear $\langle \mbox{-},\mbox{-} \rangle_X$ on the middle degree localized equivariant cohomology $\HH^{2n}_G(X)$ by
\begin{equation}\textstyle{\label{eq:bilinear}
\langle a,b \rangle_X := (-1)^n p_*(i_*)^{-1}(a \cup b),
}\end{equation}
where $i_*$ is invertible by the Localization Theorem. We can extend this idea to a product of varieties. Indeed, suppose $X_1,X_2$ are varieties of real dimension $4n_1$ and $4n_2$, respectively. We define a bilinear form $\langle \mbox{-},\mbox{-} \rangle_{X_1 \times X_2}$ on $\HH_G^{2(n_1+n_2)}(X_1 \times X_2)$ by
\begin{equation}\textstyle{\label{eq:bilinearProd}
\langle a,b \rangle_{X_1 \times X_2} := (-1)^{n_2} p_* ((i_1 \times i_2)_*)^{-1} (a \cup b),
}\end{equation}
where $i_1$ and $i_2$ are the inclusions of the $G$-fixed points into the first and second factors, respectively. An element $\alpha \in \HH^{2(n_1+n_2)}_G(X_1 \times X_2)$ then defines an operator
\begin{equation}\textstyle{\label{eq:operator}
\alpha: \HH_G^{2n_1}(X_1) \rightarrow \HH_G^{2n_2}(X_2),
}\end{equation}
by using the bilinear form to define structure constants:
\[\textstyle{ \langle \alpha(x),y \rangle_{X_2} := \langle \alpha, x \otimes y \rangle_{X_1 \times X_2}. }\]
Thus, an element $\alpha \in \HH^{2(n_1+n_2)}_G(X_1 \times X_2)$ will be called a \emph{geometric operator}.

Recall that the torus $T=(\C^*)^s \times \C^*$ acts on $\M_{\cc}(s,n)$ via
\begin{equation}\textstyle{\label{eq:Taction}
(e,t) \cdot [A,B,i,j] = [tA,t^{-1}B,ie^{-1}t^{-\cc},et^{\cc}j],
}\end{equation}
for all $(e,t) \in T$ and $[A,B,i,j] \in \M_{\cc}(s,n)$. Let $T_\bullet = \C^*$ and let $T_\bullet$ act on $\M_{\cc}(s,n)$ via the embedding
\[\textstyle{ T_\bullet \to T, \quad z \mapsto (1,z,z^2, \dots, z^{s-1},1). }\]
Similar to Remark \ref{rmk:F}, one sees that $[A,B,i,j] \in \M_{\cc}(s,n)^{T_\bullet}$ if and only if there exists a group homomorphism $g: T_\bullet \to \GL_n(\C)$ such that
\begin{equation}\textstyle{\label{eq:Tbullet}
\begin{matrix}
g(z)^{-1}Ag(z) = A, \quad g(z)^{-1}Bg(z) = B, \\[6pt]
g(z)^{-1} i = i (1,z^{-1}, \dots, z^{1-s}), \quad j g(z) = (1,z, \dots, z^{s-1}) j.
\end{matrix}
}\end{equation}
Define
\begin{equation}\textstyle{\label{eq:V^k}
V^k = V^k(A,B,i,j) := \{ v \in \C^n \; | \; g(z) v = z^{k-1} v \}.
}\end{equation}
By the stability of $(A,B,i,j)$, we have that $\C^n = \bigoplus_{k=1}^s V^k$. Moreover,
\[\textstyle{ A(V^k), B(V^k), i(\C \1_k) \subseteq V^k, \quad \text{and} \quad j(V_k) \subseteq \C \1_k,}\]
for all $k=1,\dots,s$. Conversely, if there exists a decomposition $\C^n = \bigoplus_{k=1}^s U^k$ such that $A(U^k)$, $B(U^k)$, $i(\C\1_k) \subseteq U^k$ and $j(U^k) \subseteq \C \1_k$, then we may define a group homomorphism $g: T_\bullet \to \GL_n(\C)$ by defining $g(z)|_{U^k} = z^{k-1} \id_{U^k}$. One easily checks that $g$ satisfies conditions \eqref{eq:Tbullet} and $U^k = V^k(A,B,i,j)$. Thus, $[A,B,i,j] \in \M_{\cc}(s,n)^{T_\bullet}$.

Now suppose $\nn = (\nn_1, \dots, \nn_s) \in \N^s$ such that $|\nn|=n$. Define
\[\textstyle{ \M_{\cc}(\nn) := \M_{\cc_1}(1,\nn_1) \times \cdots \times \M_{\cc_s}(1,\nn_s).}\]
Identify $\bigoplus_k \C^{\nn_k}$ with $\C^n$ by identifying $\1^k_\ell$ with $\1_{\nn_1 + \dots + \nn_{k-1} + \ell}$, where $\{\1^k_\ell\}_{\ell=1}^{\nn_k}$ is the standard basis of $\C^{\nn_k}$. An element $\left( [A_1,B_1,i_1], \dots, [A_s,B_s,i_s] \right) \in \M_{\cc}(\nn)$ then determines an element $[A,B,i,0] \in \M_{\cc}(s,n)^{T_\bullet}$ by defining
\begin{equation}\textstyle{\label{eq:Tbulletdef}
A|_{\C^{\nn_k}} := A_k, \quad B|_{\C^{\nn_k}} := B_k, \quad i := i_1 + \dots + i_s.
}\end{equation}
One can then check that we have a well-defined map
\begin{equation}\textstyle{ \label{eq:Tbulletmap}
\coprod_{|\nn|=n} \M_{\cc}(\nn) \to \M_{\cc}(s,n)^{T_\bullet}, \quad \left( [A_1,B_1,i_1], \dots, [A_s,B_s,i_s]\right)\mapsto [A,B,i,0],
}\end{equation}
where $[A,B,i,0]$ is defined as in \eqref{eq:Tbulletdef}.

\begin{lemma}\label{lem:Tbulletfixedpoints}
The map \eqref{eq:Tbulletmap} is an isomorphism of varieties.
\end{lemma}

\begin{proof}
This is a straight-forward generalization of \cite[Lemma 3.2]{nak4}.
\end{proof}

We fix once and for all an $r \in \N^+$ and a partition of $r$ of length $s$,
\[\textstyle{ \underline{r} := (r_1, \dots, r_s).}\]
Let $R' = \lcm\{r_1, \dots, r_s\}$ and define
\[\textstyle{ R := \begin{cases} R', & \text{if } R' \left( \frac{1}{r_k} + \frac{1}{r_\ell} \right) \in 2\Z \text{ for all } k, \ell, \\
2R', & \text{otherwise}. \end{cases} }\]
Consider the product variety $\M_{\cc}(\nn) = \M_{\cc_1}(1,\nn_1) \times \cdots \times \M_{\cc_s}(1,\nn_s)$. Each component $\M_{\cc_\ell}(1,\nn_\ell)$, for $\ell=1,\dots,s$, carries with it the action of a $2$-dimensional torus $T = \C^* \times \C^*$ given by setting $s=1$ in Equation \eqref{eq:Taction}. Let $T_\star = (\C^*)^s \times \C^*$ and define a $T_\star$-action on $\M_{\cc_\ell}(1,\nn_\ell)$ via the map
\[\textstyle{ T_\star \to T, \quad (e,t) \mapsto (e_\ell, t^{R/r_\ell}). }\]
That is, $T_\star$ acts on $\M_{\cc_\ell}(1,\nn_\ell)$ by
\[\textstyle{ (e,t) \star [A_\ell,B_\ell,i_\ell] = (e_\ell,t^{R/r_\ell}) \cdot [A_\ell,B_\ell,i_\ell] = [t^{R/r_\ell}A_\ell, t^{-R/r_\ell} B, i_\ell e_\ell^{-1}t^{-\cc_\ell R / r_\ell}],}\]
for all $(e,t) \in T_\star$ and $[A_\ell, B_\ell, i_\ell] \in \M_{\cc_\ell}(1,\nn_\ell)$. Then $T_\star$ acts on the product $\M_{\cc}(\nn)$ by acting on each of its components, i.e.\
\[\textstyle{ (e,t) \star ([A_1,B_1,i_1], \dots, [A_s,B_s,i_s]) = ((e,t) \star [A_1,B_1,i_1], \dots, (e,t) \star [A_s,B_s,i_s]).}\]

\begin{lemma}\label{lem:fixedpoints}
The set $\M_{\cc}(\nn)^{T_\star}$ is in one-to-one correspondence with the set
\[\textstyle{ \{ (I_1,\dots,I_s) \; | \; I_\ell \text{ is a semi-infinite monomial of charge } \cc_\ell \text{ and energy } \nn_\ell \}.}\]
\end{lemma}

\begin{proof}
We first prove that
\[\textstyle{ \M_{\cc}(\nn)^{T_\star} = \M_{\cc_1}(1,\nn_1)^T \times \cdots \times \M_{\cc_s}(1,\nn_s)^T.}\]
Let $([A_1,B_1,i_1), \dots, [A_s,B_s,i_s]) \in \M_{\cc}(\nn)^{T_\star}$. Fix $\ell \in \{1,\dots,s\}$. For all $(e,t) \in T = \C^* \times \C^*$, choose $\xi \in \C^*$ such that $\xi^{R/r_\ell} = t$. Then
\[\textstyle{ (e,t) \cdot [A_\ell,B_\ell,i_\ell] = ((1,\dots,e,\dots,1), \xi) \star [A_\ell,B_\ell,i_\ell] = [A_\ell,B_\ell,i_\ell].}\]
Therefore, $\M_{\cc}(\nn)^{T_\star} \subseteq \M_{\cc_1}(1,\nn_1)^T \times \cdots \times \M_{\cc_s}(1,\nn_s)^T$. The reverse inclusion follows by construction of the action of $T_\star$ on $\M_{\cc}(\nn)$.

Now, by \cite[Proposition 2.9]{nak2}, $\M_{\cc_\ell}(1,\nn_\ell)^T$ is in one-to-one correspondence with the set of Young diagrams of size $\nn_\ell$, which is itself in one-to-one correspondence with the set of semi-infinite monomials of charge $\cc_\ell$ and energy $\nn_\ell$ by Lemma \ref{lem:montoyoung}.
\end{proof}

In light of Lemma \ref{lem:fixedpoints}, we will henceforth identify points of $\M_{\cc}(\nn)^{T_\star}$ with $s$-tuples of semi-infinite monomials.

Define an action of $\Z_R$ on $\M_{\cc}(\nn)$ via the embedding
\[\textstyle{ \Z_R \to T_\star, \quad k \mapsto (1,\omega^k), }\]
where $\omega = e^{2 \pi \sqrt{-1} / R}$. Then $\Z_R$ acts on the $\ell$-th component, $\M_{\cc_\ell}(1,\nn_\ell)$, of $\M_{\cc}(\nn)$ via the embedding
\[\textstyle{ \Z_R \to T, \quad k \mapsto (1,\omega^{R/r_\ell}) = (1, e^{2 \pi \sqrt{-1}/r_\ell}). }\]
Thus,
\[\textstyle{ \M_{\cc}(\nn)^{\Z_R} = \M_{\cc_1}(1,\nn_1)^{\Z_{r_1}} \times \cdots \times \M_{\cc_s}(1,\nn_s)^{\Z_{r_s}},}\]
where the action of $\Z_{r_\ell}$ on $\M_{\cc_\ell}(1,\nn_\ell)$ is defined as in Equation \eqref{eq:Zmaction}. By Theorem \ref{thm:variso}, we know that $\M_{\cc_\ell}(1,\nn_\ell)^{\Z_{r_\ell}} \cong \coprod_{|\vv^\ell|= \nn_\ell} \MM(r_\ell; \vv^\ell, \1_{\bar{\cc}_\ell})$. Hence, we define
\[\textstyle{ \MM_{\cc}(\vv^1,\dots,\vv^s) := \MM (r_1; \vv^1, \1_{\bar{\cc}_1}) \times \cdots \times \MM (r_s; \vv^s, \1_{\bar{\cc}_s}),}\]
and obtain
\[\textstyle{ \M_{\cc}(\nn)^{\Z_R} \cong \coprod_{|\vv^\ell|=\nn_\ell} \MM_{\cc}(\vv^1,\dots,\vv^s).}\] 
We summarize the various fixed point varieties with the following diagram of inclusions:
\begin{equation}\textstyle{\label{eq:inclusions}
\begin{matrix}
\M_{\cc}(s,n) \supseteq \M_{\cc}(s,n)^{T_\bullet} \cong \coprod_{\nn} \M_{\cc}(\nn) \supseteq & \coprod_{\nn} \M_{\cc}(\nn)^{\Z_R} & \supseteq \coprod_{\nn} \M_{\cc}(\nn)^{T_\star}. \\
& \rotatebox{270}{$\cong \; \;$} & \\
& \bigskip \coprod_{\vv^\ell} \MM_{\cc}(\vv^1,\dots,\vv^s) &
\end{matrix}
}\end{equation}

We now consider the localized $T_\star$-equivariant cohomology of $\M_{\cc}(\nn)$. Denote the one-dimensional $T_\star$-modules
\[\textstyle{ (e,t) \mapsto e_k, \quad \text{and} \quad (e,t) \mapsto t,}\]
by $e_k$ and $t$, respectively, and denote the tensor product of such modules by juxtaposition. Moreover, we denote the first Chern classes of 
\[\textstyle{ e_k \mapsto \pt, \quad \text{and} \quad t \mapsto \pt,}\]
by $b_k$ and $\epsilon$, respectively. Thus,
\[\textstyle{ \HH_{T_\star}^*(\M_{\cc}(\nn)) = H_{T_\star}(\M_{\cc}(\nn)) \otimes_{\C[b_1,\dots,b_s,\epsilon]} \C(b_1,\dots,b_s,\epsilon).}\]
Since $\M_{\cc}(\nn)$ has real dimension $4|\nn|$, we define a bilinear form $\langle \mbox{-},\mbox{-} \rangle_{\nn,\cc}$ on $\HH^{2|\nn|}_{T_\star}(\M_{\cc}(\nn))$ as in \eqref{eq:bilinear}, induced by
\[\textstyle{ i: \M_{\cc}(s,n)^{T_\star} \hookrightarrow \M_{\cc}(\nn) \quad \text{and} \quad p: \M_{\cc}(\nn)^{T_\star} \twoheadrightarrow \pt. }\]
We extend to a bilinear form $\langle \mbox{-},\mbox{-} \rangle$ on $\bigoplus_{\nn,\cc} \HH_{T_\star}^{2|\nn|}(\M_{\cc}(\nn)) \cong \bigoplus_{n,\cc} \M_{\cc}\left(\M_{\cc}(s,n)^{T_\bullet}\right)$ by
\[\textstyle{ \langle \mbox{-},\mbox{-} \rangle := \sum_{\nn,\cc} \langle \mbox{-},\mbox{-} \rangle_{\nn,\cc}.}\]
We also define a bilinear form $\langle \mbox{-},\mbox{-} \rangle_{\nn,\mm,\cc,\dee}$ on $\HH_{T_\star}^{2(|\nn|+|\mm|)}(\M_{\cc}(\nn) \times \M_{\dee}(\mm))$ as in Equation \eqref{eq:bilinearProd}, which we extend to a bilinear form $\langle \mbox{-},\mbox{-} \rangle$ on $\bigoplus_{\nn,\mm,\cc,\dee} \HH_{T_\star}^{2(|\nn|+|\mm|)}(\M_{\cc}(\nn) \times \M_{\dee}(\mm))$ by
\[\textstyle{ \langle \mbox{-},\mbox{-} \rangle := \sum_{\nn,\mm,\cc,\dee} \langle \mbox{-},\mbox{-} \rangle_{\nn,\mm,\cc,\dee}.}\]
Our goal will be to construct geometric operators on $\bigoplus_{\nn,\cc}\HH_{T_\star}^{2|\nn|}(\M_{\cc}(\nn))$ as in Equation \eqref{eq:operator}. In order to simplify computations, it will be useful for us to introduce an orthonormal $\C(b_1,\dots,b_s,\epsilon)$-basis for $\HH_{T_\star}^*(\M_{\cc}(\nn))$. For each $\II \in \M_{\cc}(\nn)^{T_\star}$, the $T_\star$-action on $\M_{\cc}(\nn)$ induces an action on the tangent space $\T_{\II} = \T_{\II}(\M_{\cc}(\nn))$. The decomposition of $\T_{\II}$ into one-dimensional $T_\star$ modules is given in the following lemma.

\begin{lemma}
Let $\II = (\II_1, \dots, \II_s) \in \M_{\cc}(\nn)^{T_\star}$. Then, as a $T_\star$-module,
\[\textstyle{ \T_{\II} \cong \bigoplus_{\ell=1}^s \left( \bigoplus_{(i,j) \in \lambda(\II_\ell)} (t^{-h_{\lambda(\II_\ell)}(i,j) R / r_\ell} \oplus t^{h_{\lambda(\II_\ell)}(i,j) R / r_\ell}) \right),}\]
where $\lambda(\II_\ell)$ is the Young diagram associated to $\II_\ell$ and $h_{\lambda(\II)}$ is the relative hook length (see equations \eqref{eq:montoyoung} and \eqref{eq:hook}).
\end{lemma}

\begin{proof}
We have that
\[\textstyle{ \T_{\II}(\M_{\cc}(\nn)) \cong \bigoplus_{\ell=1}^s \T_{\II_\ell}(\M_{\cc_\ell}(1,\nn_\ell)).}\]
The tangent space of $\M_{\cc_\ell}(1,\nn_\ell)$ at $\II_\ell$ may then be computed by replacing $t$ by $t^{R/r_\ell}$ in \cite[Proposition 2.2]{savlic}.
\end{proof}

It will be convenient to use the decomposition $\T_{\II} = \T_{\II}^+ \oplus \T_{\II}^-$, where
\[\textstyle{ \T_{\II}^{\pm} := \bigoplus_{\ell=1}^s \left( \bigoplus_{(i,j) \in \lambda(\II_\ell)} t^{\pm h_{\lambda(\II_\ell)} (i,j) R / r_\ell} \right).}\]

\begin{lemma}
For $\II \in \M_{\cc}(\nn)^{T_\star}$, the equivariant Euler classes of $\T_{\II}^+$ and $\T_{\II}^-$ are given by
\[\textstyle{ e_{T_\star}(\T_{\II}^+) = \prod_{\ell=1}^s \left( \prod_{(i,j) \in \lambda(\II_\ell)} h_{\lambda(\II_\ell)}(i,j) \frac{R}{r_\ell} \epsilon \right),}\]
\[\textstyle{ e_{T_\star}(\T_{\II}^-) = \prod_{\ell=1}^s \left( \prod_{(i,j) \in \lambda(\II_\ell)} -h_{\lambda(\II_\ell)}(i,j) \frac{R}{r_\ell} \epsilon \right) = (-1)^{|\nn|} e_{T_\star}(\T_{\II}^+).}\]
\end{lemma}

\begin{proof}
This follows directly from the definitions of $\T_{\II}^+$ and $\T_{\II}^-$.
\end{proof}

For each $\II \in \M_{\cc}(\nn)^{T_\star}$, let
\[\textstyle{ [\II] := \frac{i_*(1_{\II})}{e_{T_\star}(\T_{\II}^-)} \in \HH^{2|\nn|}_{T_\star}(\M_{\cc}(\nn)),}\]
where $1_{\II}$ is the unit in $\HH_{T_\star}^*(\pt)$ and $e_{T_\star}(\T_{\II}^-)$ is to be interpreted as an invertible element in this ring. Since the elements $1_{\II}$ form a $\C(b_1, \dots, b_s, \epsilon)$-basis of $\bigoplus_{\nn,\cc} \HH_{T_\star}^*\left((\M_{\cc}(\nn))^{T_\star}\right)$, by the Localization Theorem (Theorem \ref{thm:localization}), the elements $[\II]$ form a $\C(b_1,\dots,b_s,\epsilon)$-basis of $\bigoplus_{\nn,\cc} \HH_{T_\star}^*(\M_{\cc}(\nn))$.

\begin{lemma}\label{lem:orthonormal}
The $[\II]$ are orthonormal with respect to the bilinear form $\langle \mbox{-},\mbox{-} \rangle$.
\end{lemma}

\begin{proof}
The proof is completely analogous to \cite[Proposition 2.4]{savlic}.
\end{proof}

Define
\begin{equation}\textstyle{\label{eq:A}
\A = \vspan_{\C} \left\{ [\II] \; | \; \II \in (\M_{\cc}(\nn))^{T_\star}, \; \nn \in \N^s, \; \cc \in \Z^s \right\}.
}\end{equation}
Then $\A$ is a full $\C$-lattice in $\bigoplus_{\nn,\cc} \HH^*_{T_\star}(\M_{\cc}(\nn))$. One has the $(\N \times \Z)$-grading on $\A$:
\[\textstyle{ \A = \bigoplus_{\nn,\cc} \A_{\cc}(\nn), \quad \A_{\cc}(\nn) = \vspan_{\C} \left\{ [\II] \; | \; \II \in \M_{\cc}(\nn)^{T_\star} \right\}.}\]
It will also be convenient for us to use the following $\Z$-grading on $\A$:
\[\textstyle{ \A = \bigoplus_{c \in \Z} \A(c), \quad \A(c) := \vspan_{\C} \left\{ [\II] \in \A_{\cc}(\nn) \; | \; |\cc| = c \right\}.}\]

\begin{cor}
The restriction of $\langle \mbox{-},\mbox{-} \rangle$ to $\A$ is non-degenerate and $\C$-valued.
\end{cor}

\begin{proof}
This follows directly from Lemma \ref{lem:orthonormal}.
\end{proof}

\begin{rmk}\label{rmk:AviaKunneth}
Notice that via the K\"unneth formula
\[\textstyle{ \HH_{T_\star}^*(\M_{\cc}(\nn) \cong \HH_{T_\star}^*(\M_{\cc_1}(1,\nn_1)) \otimes \cdots \otimes \HH_{T_\star}^*(\M_{\cc_s}(1,\nn_s)),}\]
the element $[\II] \in \HH_{T_\star}^*(\M_{\cc}(\nn))$, where $\II = (\II_1,\dots,\II_s)$, maps to
\[\textstyle{ [\II_1] \otimes \cdots \otimes [\II_s] \in \HH_{T_\star}^*(\M_{\cc_1}(1,\nn_1)) \otimes \cdots \otimes \HH_{T_\star}^*(\M_{\cc_s}(1,\nn_s)).}\]
Therefore,
\[\textstyle{ \A_{\cc}(\nn) \cong \A_{\cc_1}(\nn_1) \otimes \cdots \otimes \A_{\cc_s}(\nn_s),}\]
where $\A_{\cc_\ell}(\nn_\ell) = \{ [\II_\ell] \; | \; \II_\ell  \in \M_{\cc_\ell}(1,\nn_\ell)^{T_\star} \}.$
\end{rmk}

Let $X$ be an algebraic variety with a $T_\star$-action, and let $E \to X$ be a $T_\star$-equivariant vector bundle. We will denote the $k$-th equivariant Chern class of $E$ by $c_k(E)$. Note that $c_k(E) \in H_{T_\star}^{2k}(X)$. The following lemma will act as our main tool in constructing geometric operators in the following section

\begin{lemma}\label{lem:operator}
Let $\II \in \M_{\cc}(\nn)^{T_\star}$ and $\JJ \in \M_{\dee}(\mm)^{T_\star}$, let $E$ be an equivariant vector bundle on $\M_{\cc}(\nn) \times \M_{\dee}(\mm)$ and let $\beta \in \HH^{2k}_{T_\star}(\M_{\cc}(\nn) \times \M_{\dee}(\mm))$. Then
\[\textstyle{ \langle \beta \cup c_{|\nn|+|\mm|-k}(E) [\II], [\JJ] \rangle = \frac{\beta_{\II,\JJ} \cup c_{|\nn|+|\mm|-k}(E_{(\II,\JJ)})}{e_{T_\star}(\T_{\II}^-) e_{T_\star}(\T_{\JJ}^+)},}\]
where $c_{n+m-k}(E_{(\II,\JJ)}) \in H^*_{T_\star}(\pt) = \C[b_1,\dots,b_s,\epsilon]$ is the polynomial given by the equivariant Chern class of the fibre $E$ over the point $(\II,\JJ)$ and $\beta_{\II,\JJ} = i^*_{\II,\JJ}(\beta)$, where $i_{\II,\JJ}: \{(\II,\JJ)\} \hookrightarrow \M_{\cc}(\nn) \times \M_{\dee}(\mm)$ is the inclusion of the fixed point.
\end{lemma}

\begin{proof}
See \cite[Lemma 2.6]{savlic}.
\end{proof}

\begin{rmk}
The precise statement of \cite[Lemma 2.6]{savlic} differs slightly from ours (since we use the variety $\M_{\cc}(\nn)$ under the action of $T_\star$ rather than $\M_{\cc}(s,n)$ under the action of $T$). However, every step in the proof of \cite[Lemma 2.6]{savlic} applies to Lemma \ref{lem:operator}; thus the two lemmas are essentially the same.
\end{rmk}

The next step will be to construct $T_\star$-equivariant vector bundles over $\M_{\cc}(\nn) \times \M_{\dee}(\mm)$ whose Chern classes will define the appropriate geometric operators (as our ultimate goal is to construct geometric versions of the Heisenberg, Clifford and Chevalley operators from Section \ref{section:1}). We begin by defining vector bundles
\[\textstyle{ \C^n \times_{\GL_n(\C)} M_0^{\st}(s,n) \to \M_{\cc}(s,n), \quad \text{and} \quad \C^s \times \M_{\cc}(s,n) \to \M_{\cc}(s,n),}\]
which we denote by $\V = \V(\cc,s,n)$ and $\W = \W(\cc,s,n)$, respectively. Note that $\V$ and $\W$ are simply the associated bundles of the trivial $\GL_n(\C)$-bundles
\[\textstyle{ \C^n \times M_0^{\st}(s,n) \to M_0^{\st}(s,n), \quad \text{and} \quad \C^s \times M_0^{\st}(s,n) \to M_0^{\st}(s,n).}\]
The bundles $\V$ and $\W$ are $T$-equivariant with respect to the trivial action of $T$ on $\C^n$ and the natural action of $T$ on $\C^s$, respectively. Consider the Hom-bundle $\Hom(\V,\V)$ on $\M_{\cc}(s,n)$. We can define a global section $s: \M_{\cc}(s,n) \to \Hom(\V,\V)$ by defining $s[A,B,i,j]$ to be the (well-defined) linear map
\begin{align*}
\textstyle{ \C^n \times_{\GL_n(\C)} [A,B,i,j]} &\to \textstyle{\C^n \times_{\GL_n(\C)} [A,B,i,j],} \\
\textstyle{\GL_n(\C) \cdot (v, (A,B,i,j))} &\mapsto \textstyle{\GL_n(\C) \cdot (Av, (A,B,i,j)).}
\end{align*}
By a slight abuse of notation, we will denote the section $s$ by $A$. We similarly define sections $B$, $i$ and $j$ of $\Hom(\V,\V)$, $\Hom(\W,\V)$ and $\Hom(\V,\W)$, respectively.

One can extend this construction to a product of moduli spaces. The bundle $\V(\cc,s,n) \to \M_{\cc}(s,n)$ may be extended to a vector bundle over the product $\M_{\cc}(s,n) \times \M_{\dee}(s,m)$ by:
\[\textstyle{ \V(\cc,s,n) \times \M_{\dee}(s,m) \to \M_{\cc}(s,n) \times \M_{\dee}(s,m).}\]
We denote this vector bundle by $\V_1 = \V_1(\cc,\dee,s,n,m)$. Likewise, we extend the bundles $\W(\cc,s,n) \to \M_{\cc}(s,n)$, $\V(\dee,s,m) \to \M_{\dee}(s,m)$ and $\W(\dee,s,m) \to \M_{\dee}(s,m)$ to bundles $\W_1 = \W_1(\cc,\dee,s,n,m)$, $\V_2 = \V_2(\cc,\dee,s,n,m)$ and $\W_2 = \W_2(\cc,\dee,s,n,m)$ over $\M_{\cc}(s,n) \times \M_{\dee}(s,m)$. We then have bundles $\Hom(\V_k, \V_k)$, $\Hom(\W_k,\V_k)$ and $\Hom(\V_k, \W_k)$ as vector bundles over $\M_{\cc}(s,n) \times \M_{\dee}(s,m)$ with sections $A_k$, $B_k$, $i_k$ and $j_k$, where $k=1,2$. We define a three-term, $T$-equivariant complex of vector bundles on $\M_{\cc}(s,n) \times \M_{\dee}(s,m)$ by
\begin{equation}\textstyle{\label{eq:vbcomplex}
\Hom(\V_1, \V_2) \xrightarrow{\zeta} \begin{matrix} t \Hom(\V_1, \V_2) \oplus t^{-1} \Hom(\V_1, \V_2) \\ \oplus \\ \Hom(\W_1, \V_2) \oplus \Hom(\V_1, \W_2) \end{matrix} \xrightarrow{\tau} \Hom(\V_1, \V_2),
}\end{equation}
where
\[\textstyle{ \zeta(X) = \begin{pmatrix} X A_1 - A_2 X \\ X B_1 - B_2 X \\ X i_1 \\ -j_2 X \end{pmatrix}, \quad \text{and} \quad \tau \begin{pmatrix} C \\ D \\ a \\ b \end{pmatrix} = A_2 D - D A_1 + C B_1 - B_2 C + i_2 b + a j_1.}\]
Note that here the modules $t^{\pm 1}$ are the one-dimensional $T$-modules $(e,t) \mapsto t^{\pm 1}$. One verifies that $\tau \circ \zeta = 0$.

\begin{lemma}
The cohomology $\ker \tau / \im \zeta$ of Complex \eqref{eq:vbcomplex} is a vector bundle on $\M_{\cc}(s,n) \times \M_{\dee}(s,m)$.
\end{lemma}

\begin{proof}
See \cite[Lemma 3.2]{savlic}.
\end{proof}

For $n,m \in \N$, we will denote the vector bundle 
\[\textstyle{ \ker \tau / \im \zeta \to \M_{\cc}(s,n) \times \M_{\dee}(s,m),}\]
by $\K_{\cc,\dee}(s,n,m)$. Notice that, by construction, one has the following vector bundle on $\M_{\cc}(s,n)^{T_\bullet} \times \M_{\dee}(s,n)^{T_\bullet}$:
\[\textstyle{ (\ker \tau / \im \zeta)^{T_\bullet} \to \M_{\cc}(s,n)^{T_\bullet} \times \M_{\dee}(s,n)^{T_\bullet},}\]
which we denote by $\K_{\cc,\dee}(s,n,m)^{T_\bullet}$. By Lemma \ref{lem:Tbulletfixedpoints},
\[\textstyle{ \M_{\cc}(s,n)^{T_\bullet} \times \M_{\dee}(s,m)^{T_\bullet} \cong \coprod_{|\nn|=n, \; |\mm|=m} \M_{\cc}(\nn) \times \M_{\dee}(\mm),}\]
and so we may identify these varieties and consider the restriction of $\K_{\cc,\dee}(s,n,m)^{T_\bullet}$:
\[\textstyle{ \K_{\cc,\dee}(\nn,\mm) := \K_{\cc,\dee}(s,n,m)^{T_\bullet}|_{\M_{\cc}(\nn) \times \M_{\dee}(\mm)},}\]
which is a vector bundle on $\M_{\cc}(\nn) \times \M_{\dee}(\mm)$. On $\M_{\cc_\ell}(1,\nn_\ell) \times \M_{\dee_\ell}(1,\mm_\ell)$, the product of the $\ell$-th components of $\M_{\cc}(\nn)$ and $\M_{\dee}(\mm)$, one has the vector bundle $\K_{\cc_\ell,\dee_\ell}(1,\nn_\ell,\mm_\ell)$. Let
\[\textstyle{ f_\ell: \M_{\cc}(\nn) \times \M_{\dee}(\mm) \to \M_{\cc_\ell}(1,\nn_\ell) \times \M_{\dee_\ell}(1,\mm_\ell),}\]
denote the canonical projection. Then the vector bundle pullback, $f_\ell^*\K_{\cc_\ell,\dee_\ell}(1,\nn_\ell,\mm_\ell)$, is a vector bundle on the full product $\M_{\cc}(\nn) \times \M_{\dee}(\mm)$.

\begin{lemma}\label{lem:vbpullback}
There is an isomorphism of vector bundles
\[\textstyle{ \K_{\cc,\dee}(\nn,\mm) \cong \bigoplus_{\ell=1}^s f_\ell^*(\K_{\cc_\ell,\dee_\ell}(1,\nn_\ell,\mm_\ell)).}\]
\end{lemma}

\begin{proof}
The result follows by computing the $T_\bullet$-fixed points of Complex \eqref{eq:vbcomplex} over $\M_{\cc,}(\nn)^{T_\bullet} \times \M_{\dee}(\mm)^{T_\bullet}$.
\end{proof}

From our diagram of inclusions, \eqref{eq:inclusions}, we have that
\begin{align*}
\textstyle{\M_{\cc}(\nn) \times \M_{\dee}(\mm)} &\textstyle{\supseteq \M_{\cc}(\nn)^{\Z_R} \times \M_{\dee}(\mm)^{\Z_R}}  \\
&\textstyle{\cong \coprod_{{|\vv^\ell|=\nn_\ell} \atop {|\uu^\ell|=\mm_\ell}} \MM_{\cc}(\vv^1,\dots, \vv^\ell) \times \MM_{\dee}(\uu^1,\dots,\uu^s).
}\end{align*}
We may thus view $\MM_{\cc}(\vv^1, \dots, \vv^s) \times \MM_{\dee}(\uu^1, \dots, \uu^s)$ as a subvariety of $\M_{\cc}(\nn) \times \M_{\dee}(\mm)$, and consider the restriction of $\K_{\cc,\dee}(\nn,\mm)$: 
\[\textstyle{ \KK_{\cc,\dee}(\vv^1,\uu^1,\dots,\vv^s,\uu^s) := \K_{\cc,\dee}(\nn,\mm)|_{\MM_{\cc}(\vv^1, \dots, \vv^s) \times \MM_{\dee}(\uu^1,\dots, \uu^s)},}\]
which is a vector bundle on $\MM_{\cc}(\vv^1,\dots,\vv^s) \times \MM_{\dee}(\uu^1,\dots,\uu^s)$. In the following section, the vector bundles $\K_{\cc,\dee}(\nn,\mm)$ and $\KK_{\cc,\dee}(\vv^1,\uu^1,\dots,\vv^s,\uu^s)$ will allows us to construct geometric versions of the Heisenberg algebra, the Clifford algebra and $\widehat\gl_r$.

%
%

\section{Geometric Realizations of $V_\text{basic}$}\label{section:4}

In this section, we present the main theorem of this paper, Theorem \ref{thm:main}, which describes our geometric realizations of the basic representation of $\widehat\gl_r$. We will do so by constructing the oscillator algebra, the Clifford algebra and $\widehat\gl_r$ as geometric operators (in the sense of \eqref{eq:operator}). Using the Localization Theorem, we will consider these geometric operators as operators on the same cohomology. We will show that these operators satisfy the same relations as those observed by their algebraic counterparts in Lemma \ref{lem:glaction}, and that the cohomology on which they act (or more specifically the full $\C$-lattice $\A$) corresponds naturally to fermionic Fock space. 

Throughout this section, for any $T_\star$-equivariant vector bundle $E$, we let $c_{\tnv}(E)$ denote the \emph{top nonvanishing} $T_\star$-equivariant Chern class of $E$. We will also frequently make use of the K\"unneth formula
\[\textstyle{ \HH_{T_\star}^*(X \times Y) \cong \HH_{T_\star}^*(X) \otimes \HH_{T_\star}^*(Y),}\]
and simply identify these two rings.

We begin by constructing the geometric version of the oscillator/Heisenberg algebra. Consider the vector bundle $\K_{\cc,\dee}(\nn,\mm)$ over $\M_{\cc}(\nn) \times \M_{\dee}(\mm)$. The rank of $\K_{\cc,\dee}(\nn,\mm)$ is $|\nn|+|\mm|$, which can be seen from the following lemma.

\begin{lemma}\label{lem:Euler}
Let $(\II,\JJ) \in \M_{\cc}(\nn)^{T_\star} \times \M_{\dee}(\mm)^{T_\star}$. The equivariant Euler class of the restriction of $\K_{\cc,\dee}(\nn,\mm)$ to $(\II,\JJ)$ is
\begin{align*}
\textstyle{ e_{T_\star}(\K_{\cc,\dee}(\nn,\mm)_{(\II,\JJ)}) = \prod_{\ell=1}^s} &\textstyle{\left(\prod_{x \in \lambda(\II_\ell)} (\dee_\ell - \cc_\ell - h_{\lambda(\II_\ell),\lambda(\JJ_\ell)}(x)) \frac{R}{r_\ell} \epsilon \right)} \\ & \quad \textstyle{ \left( \prod_{y \in \lambda(\JJ_\ell)} (\dee_\ell - \cc_\ell + h_{\lambda(\JJ_\ell),\lambda(\II_\ell)}(y)) \frac{R}{r_\ell} \epsilon \right).}
\end{align*}
\end{lemma}

\begin{proof}
By Lemma \ref{lem:vbpullback},
\[\textstyle{ \K_{\cc,\dee}(\nn,\mm) \cong \bigoplus_{\ell=1}^s f_\ell^*(\K_{\cc_\ell,\dee_\ell}(1,\nn_\ell,\mm_\ell)),}\]
and so, by properties of the Euler class,
\[\textstyle{ e_{T_\star}(\K_{\cc,\dee}(\nn,\mm)) = \prod_{\ell=1}^s f_\ell^*(e_{T_\star}(\K_{\cc_\ell,\dee_\ell}(1,\nn_\ell,\mm_\ell)) = \prod_{\ell=1}^s \left( 1^{\otimes \ell-1} \otimes e_{T_\star}(\K_{\cc_\ell,\dee_\ell}(1, \nn_\ell,\mm_\ell)) \otimes 1^{\otimes s-\ell} \right).}\]
Now, by setting $r=1$ and replacing $\epsilon$ by $(R/r_\ell)\epsilon$ in \cite[Lemma 3.3]{savlic}, we have that
\begin{align*}
\textstyle{ e_{T_\star}(\K_{\cc_\ell,\dee_\ell}(1,\nn_\ell,\mm_\ell)_{(\II_\ell,\JJ_\ell)}) =} & \textstyle{\left(\prod_{x \in \lambda(\II_\ell)} (\dee_\ell - \cc_\ell - h_{\II_\ell,\JJ_\ell}(x)) \frac{R}{r_\ell} \epsilon \right)} \\ &\quad \textstyle{ \left( \prod_{y \in \lambda(\JJ_\ell)} (\dee_\ell - \cc_\ell + h_{\JJ_\ell,\II_\ell}(y)) \frac{R}{r_\ell} \epsilon \right).} \qedhere
\end{align*}
\end{proof}

We recall from \cite[Section 3.3]{savlic}, that the top nonvanishing Chern class of $\K_{c,c}(1,n,m)$ is
\begin{equation}\textstyle{\label{eq:HeisenChern}
c_{\tnv}(\K_{c,c}(1,n,m)) = \begin{cases} c_{2n}(\K_{c,c}(1,n,n)), & \text{if } n=m, \\ c_{n+m-1}(\K_{c,c}(1,n,m)) & \text{if } n \ne m. \end{cases}
}\end{equation}
By Lemma \ref{lem:vbpullback}, the top nonvanishing Chern class of $\K_{\cc,\cc}(\nn,\mm)$ may be computed by
\begin{align*}
\textstyle{c_{\tnv}(\K_{\cc,\cc}(\nn,\mm))} &\textstyle{= c_{\tnv}\left( \bigoplus_{\ell=1}^s f_\ell^*(\K_{\cc_\ell,\cc_\ell}(1,\nn_\ell,\mm_\ell))\right) = \prod_{\ell=1}^s f_\ell^*(c_{\tnv}(\K_{\cc_\ell, \cc_\ell}(1,\nn_\ell,\mm_\ell)))} \\
&\textstyle{= \prod_{\ell=1}^s \left(1^{\otimes \ell-1} \otimes c_{\tnv}(\K_{\cc_\ell,\cc_\ell}(1,\nn_\ell,\mm_\ell)) \otimes 1^{\otimes s-\ell}\right).
}\end{align*}
Therefore, if $\nn=\mm$,
\[\textstyle{ c_{\tnv}(\K_{\cc,\cc}(\nn,\nn)) = c_{2|\nn|}(\K_{\cc,\cc}(\nn,\nn)), }\]
whereas if $\nn$ differs from $\mm$ in exactly one component,
\[\textstyle{ c_{\tnv}(\K_{\cc,\cc}(\nn,\mm)) = c_{|\nn|+|\mm|-1}(\K_{\cc,\cc}(\nn,\mm)). }\]

Define $\alpha \in \HH_{T_\star}(\M_{\cc}(\nn)^{T_{\star}} \times \M_{\dee}(\mm)^{T_\star}) = \bigoplus_{(\II,\JJ)} \HH_{T_\star}(\pt)$ to be the element with $(\II,\JJ)$-th component
\[\textstyle{ \alpha_{(\II,\JJ)} = \frac{\epsilon}{e_{T_\star}(\T_{(\II,\JJ)})},}\]
where $\T_{(\II,\JJ)}$ is the tangent space of $(\II,\JJ)$ in $\M_{\cc}(\nn) \times \M_{\dee}(\mm)$. Let $\gamma := i_*(\alpha)$, where $i: \M_{\cc}(\nn)^{T_\star} \times \M_{\dee}(\mm)^{T_\star} \hookrightarrow \M_{\cc}(\nn) \times \M_{\dee}(\mm)$ is the natural inclusion. If we denote by $i_{\II,\JJ}$ the inclusion $\{(\II,\JJ)\} \hookrightarrow \M_{\cc}(\nn) \times \M_{\dee}(\mm)$, then by \cite[Equation (9.3)]{cdks},
\[\textstyle{ i^*_{\II,\JJ} (\gamma) = (i^*_{\II,\JJ} \circ i_*) (\alpha) = e_{T_\star}(\T_{(\II,\JJ)}) \cup \alpha_{(\II,\JJ)} = \epsilon.}\]
Note that $\gamma$ and $\epsilon$ are elements of degree $2$ in their respective cohomologies.

\begin{defn}[Geometric oscillator/Heisenberg operators]\label{defn:geoHeisenberg}
For $\ell=1,\dots,s$ and $k \in \Z$, define operators
\[\textstyle{ \geoP_\ell(k): \bigoplus_{\nn,\cc} \HH_{T_\star}^{2|\nn|}(\M_{\cc}(\nn)) \to \bigoplus_{\nn,\cc} \HH_{T_\star}^{2|\nn|}(\M_{\cc}(\nn)), }\]
by
\[\textstyle{ \geoP_\ell(k)|_{\HH_{T_\star}^{2|\nn|}(\M_{\cc}(\nn))} = \begin{cases} (R/r_\ell)\gamma \cup c_{\tnv}(\K_{\cc,\cc}(\nn,\nn-k\1_\ell)), & \text{if } k<0, \\ - (R/r_\ell)\gamma \cup c_{\tnv}(\K_{\cc,\cc}(\nn,\nn-k\1_\ell)), & \text{if } k>0,\end{cases}}\]
\[\textstyle{ \qquad \qquad \in \HH_T^{2(2|\nn|-k)}(\M_{\cc}(\nn) \times \M_{\cc}(\nn-k\1_\ell))}\]

\smallskip

\[\textstyle{ \geoP_\ell(0)|_{\HH_{T_\star}^{2|\nn|}(\M_{\cc}(\nn))} = \cc_\ell \cdot c_{\tnv}(\K_{c,c}(\nn,\nn)) = \cc_\ell \id.}\]
The $\geoP_\ell(k)$ will be called \emph{geometric oscillators} (or, for $k\ne 0$, \emph{geometric Heisenberg operators}). 
\end{defn}

\begin{thm}\label{thm:geoHeisenberg}
The operators $\geoP_\ell(k)$ preserve the space $\A$ and satisfy the commutation relations
\[\textstyle{ [\geoP_\ell(k), \geoP_m(0)]=0, \quad \text{and} \quad [\geoP_\ell(k),\geoP_m(j)] = \frac{1}{k} \delta_{\ell, m} \delta_{k+j,0} \id, \; k \ne 0. }\]
In particular, the mapping
\[\textstyle{ P_\ell(k) \mapsto \geoP_\ell(k),}\]
defines a representation of the $s$-coloured oscillator algebra on $\A$ and the linear map determined by
\[\textstyle{ [\II] \mapsto \left(s_{\lambda(\II_1)} \otimes q^{c(\II_1)} \right) \otimes \cdots \otimes \left(s_{\lambda(\II_s)} \otimes q^{c(\II_s)} \right), }\]
is an isomorphism of $s$-coloured oscillator algebra representations $\A \to \BB$. This isomorphism maps $\A(c) \to \BB(c)$, for all $c \in \Z$.
\end{thm}

\begin{proof}
Take $\mm = \nn - k \1_\ell$. We know that
\[\textstyle{ c_{\tnv}(\K_{\cc,\cc}(\nn,\mm)) = \prod_{i=1}^s \left(1^{\otimes i-1} \otimes c_{\tnv}(\K_{\cc_i,\cc_i}(1,\nn_i,\nn_i)) \otimes 1^{\otimes s-i}\right). }\]
By \cite[Lemma 3.10]{savlic}, for $i \ne \ell$, we have that $c_{\tnv}(\K_{\cc_i,\cc_i}(1,\nn_i, \nn_i))$ is the identity as an operator $\HH_{T_\star}^{2\nn_i}(\M_{\cc_i}(1,\nn_i)) \to \HH_{T_\star}^{2\nn_i}(\M_{\cc_i}(1,\nn_i))$, i.e.\ $c_{\tnv}(\K_{\cc_i, \cc_i}(1,\nn_i,\mm_i)) = 1$. Thus,
\[\textstyle{ c_{\tnv}(\K_{\cc,\cc}(\nn,\mm)) = 1^{\otimes \ell-1} \otimes c_{\tnv}(\K_{\cc_\ell,\cc_\ell}(1,\nn_\ell, \nn_\ell - k)) \otimes 1^{\otimes s-\ell}.}\]
Therefore, via the K\"unneth formula, 
\[\textstyle{ \bigoplus_{\nn,\cc} \HH_{T_\star}^{2|\nn|}(\M_{\cc}(\nn)) \cong \bigoplus_{\nn,\cc} \HH_{T_\star}^{2\nn_1}(\M_{\cc_1}(1,\nn_1)) \otimes \cdots \otimes \HH_{T_\star}^{2\nn_s}(\M_{\cc_s}(1,\nn_s)),}\]
we have
\[\textstyle{ \geoP_\ell(k)|_{\HH_{T_\star}^{2|\nn|}(\M_{\cc}(\nn))} = 1^{\otimes(\ell-1)} \otimes (R/r_\ell) \gamma \cup c_{\tnv}(\K_{\cc_\ell,\cc_\ell}(1,\nn_\ell,\nn_\ell-k)) \otimes 1^{\otimes(s-\ell)}.}\]
By \cite[Theorem 3.14]{savlic}, for a fixed $\ell \in \{1,\dots,s\}$, the operators $\geoP_\ell(k)$ preserve the space $\bigoplus_{\nn_\ell,\cc_\ell} \A_{\cc_\ell}(\nn_\ell)$ (see Remark \ref{rmk:AviaKunneth}) and satisfy the $1$-coloured oscillator algebra relations. The general result then follows by extension to the whole tensor product.
\end{proof}

For the geometric version of the Clifford algebra, we recall from \cite[Section 3.2]{savlic} that
\begin{equation}\textstyle{\label{eq:CliffChern}
c_{\tnv}(\K_{c,c\pm 1}(1,n,m)) = c_{n+m}(\K_{c,c \pm 1}(1,n,m)).
}\end{equation}
Therefore, by equations \eqref{eq:HeisenChern} and \eqref{eq:CliffChern}, if $\nn_i = \mm_i$ for all $i \ne \ell$, then
\[\textstyle{ c_{\tnv}(\K_{\cc,\cc \pm \1_\ell}(\nn,\mm)) = c_{|\nn|+|\mm|}(\K_{\cc,\cc \pm \1_\ell}(\nn,\mm)).}\]

\begin{defn}[Geometric Clifford operators]
For $\ell=1,\dots,s$ and $k \in \Z$, define operators
\[\textstyle{ \geoPsi_\ell(k), \geoPsi_\ell^*(k): \bigoplus_{\nn,\cc} \HH_{T_\star}^{2|\nn|}(\M_{\cc}(\nn)) \rightarrow \bigoplus_{\nn,\cc} \HH_{T_\star}^{2|\nn|}(\M_{\cc}(\nn)),}\]
by
\begin{align*}
\textstyle{\geoPsi_\ell(k)|_{\HH_{T_\star}^{2|\nn|}(\M_{\cc}(\nn))}} &\textstyle{:= (-1)^{\cc_1 + \dots + \cc_{\ell-1}} c_{\tnv}(\K_{\cc,\cc+\1_\ell}(\nn,\nn+(k-\cc_\ell-1) \1_\ell))} \\
& \textstyle{\in \HH_{T_\star}^{2(2|\nn|+k-\cc_\ell-1)}(\M_{\cc}(\nn) \times \M_{\cc + \1_\ell}(\nn + (k-\cc_\ell-1)\1_\ell)),}
\end{align*}
\begin{align*}
\textstyle{\geoPsi^*_\ell(k)|_{\HH_{T_\star}^{2|\nn|}(\M_{\cc}(\nn))}} &\textstyle{:= (-1)^{\cc_1 + \dots + \cc_{\ell-1}}c_{\tnv}(\K_{\cc,\cc-\1_\ell}(\nn,\nn-(k-\cc_\ell) \1_\ell))} \\
& \textstyle{\in \HH_{T_\star}^{2(2|\nn|-k+\cc_\ell)}(\M_{\cc}(\nn) \times \M_{\cc - \1_\ell}(\nn - (k-\cc_\ell)\1_\ell)).}
\end{align*}
The $\geoPsi_\ell(k)$ and $\geoPsi^*_\ell(k)$ will be called \emph{geometric Clifford operators}.
\end{defn}

\begin{thm}\label{thm:geoClifford}
The operators $\geoPsi_\ell(k)$ and $\geoPsi_\ell^*(k)$ preserve the space $\A$ and satisfy the relations
\[\textstyle{ \{ \geoPsi_\ell(k), \geoPsi^*_j(i) \} = \delta_{ki} \delta_{\ell j}, \quad \{\geoPsi_\ell(k), \geoPsi_j(i)\} = \{ \geoPsi^*_\ell(k), \geoPsi^*_j(i) \} = 0.}\]
In particular, the mapping
\[\textstyle{ \psi_\ell(k) \mapsto \geoPsi_\ell(k), \quad \psi_\ell^*(k) \mapsto \geoPsi^*_\ell(k),}\]
defines a representation of the $s$-coloured Clifford algebra on $\A$ and the linear map determined by $[\II] \mapsto \II,$ is an isomorphism of Clifford algebra representations $\A \to \FF$. This isomorphism maps $\A(c) \to \FF(c)$, for all $c \in \Z$.
\end{thm}

\begin{proof}
Using an argument analogous to the proof of Theorem \ref{thm:geoHeisenberg}, one can show that
\[\textstyle{ \geoPsi_\ell(k)|_{\HH_{T_\star}(\M_{\cc}(\nn))} = (-1)^{\cc_1 + \dots + \cc_{\ell-1}} (1^{\otimes \ell - 1} \otimes c_{\tnv}(\K_{\cc_\ell,\cc_\ell + 1}(1,\nn_\ell,\nn_\ell + k - \cc_\ell - 1)) \otimes 1^{\otimes s - \ell}),}\]
\[\textstyle{ \geoPsi^*_\ell(k)|_{\HH_{T_\star}(\M_{\cc}(\nn))} = (-1)^{\cc_1 + \dots + \cc_{\ell-1}} (1^{\otimes \ell - 1} \otimes c_{\tnv}(\K_{\cc_\ell,\cc_\ell - 1}(1, \nn_\ell, \nn_\ell - k + \cc_\ell)) \otimes 1^{\otimes s - \ell}).}\]
By \cite[Theorem 3.6]{savlic}, for fixed $\ell$, the operators $(-1)^{\cc_1 + \dots + \cc_{\ell-1}} c_{\tnv}(\K_{\cc_\ell,\cc_\ell + 1}(1,\nn_\ell,\nn_\ell + k - \cc_\ell - 1))$ and $(-1)^{\cc_1 + \dots + \cc_{\ell-1}} c_{\tnv}(\K_{\cc_\ell,\cc_\ell - 1}(1, \nn_\ell, \nn_\ell - k + \cc_\ell))$ preserve the space $\bigoplus_{\nn_\ell, \cc_\ell} \A(\nn_\ell, \cc_\ell)$ (see Remark \ref{rmk:AviaKunneth}) and satisfy the $1$-coloured Clifford algebra relations. The general result then follows by extension to the whole tensor product.
\end{proof}

Recall from Diagram \ref{eq:inclusions} that, since
\[\textstyle{ \M_{\cc}(\nn)^{\Z_R} \cong \coprod_{|\vv^\ell|=n_\ell} \MM_{\cc}(\vv^1,\dots, \vv^s), }\]
we may view $\MM_{\cc}(\vv^1, \dots, \vv^s) \times \MM_{\dee}(\uu^1, \dots, \uu^s)$ as a subvariety of $\M_{\cc}(\nn) \times \M_{\dee}(\mm)$, where $\nn_\ell = |\vv^\ell|$ and $\mm_\ell = |\uu^\ell|$. We therefore have an action of $T_\star$ on $\M_{\cc}(\vv^1,\dots, \vv^s)$ and
\[\textstyle{ \coprod_{\vv^\ell} \MM_{\cc}(\vv^1,\dots,\vv^s)^{T_\star} = \coprod_{\nn} \M_{\cc}(\nn)^{T_\star}.}\]
Let
\[\textstyle{ \KK_{\cc}(\nn; \ell, k)^{\pm} := \coprod_{|\vv^i|=\nn_i} \KK_{\cc,\cc}(\vv^1,\vv^1,\dots,\vv^\ell,\vv^\ell \pm \1_k, \dots, \vv^s,\vv^s),}\]
which is a vector bundle on $\coprod_{|\vv^i| = \nn_i}\MM_{\cc}(\vv^1,\dots,\vv^s) \times \MM_{\cc}(\vv^1,\dots,\vv^\ell \pm \1_k,\dots,\vv^s)$. By construction, for any $(\II,\JJ) \in \MM_{\cc}(\vv^1,\dots,\vv^s)^{T_\star} \times \MM_{\cc}(\vv^1, \dots, \vv^\ell \pm \1_k, \dots, \vv^\ell)^{T_\star}$,
\[\textstyle{ c_{\tnv} (\KK_{\cc}(\nn; \ell,k)^{\pm}_{(\II,\JJ)}) = c_{\tnv} (\K_{\cc,\cc}(\nn,\nn \pm \1_\ell)_{(\II,\JJ)}).}\]
Thus,
\[\textstyle{ c_{\tnv}(\KK_{\cc}(\nn; \ell,k)^{\pm}) = c_{(2|\nn| \pm 1)-1}(\KK_{\cc}(\nn; \ell, k)^{\pm}).}\]
By the same reasoning, for $(\II,\JJ) \in \MM_{\cc}(\vv^1,\dots,\vv^s)^{T_\star} \times \MM_{\cc}(\vv^1,\dots,\vv^s)^{T_\star}$,
\[\textstyle{ c_{\tnv}(\KK_{\cc,\cc}(\vv^1,\vv^1,\dots,\vv^s,\vv^s)_{(\II,\JJ)}) = c_{\tnv}(\K_{\cc,\cc}(\nn,\nn)_{(\II,\JJ)}),}\]
and so
\[\textstyle{ c_{\tnv}(\KK_{\cc,\cc}(\vv^1,\vv^1,\dots,\vv^s,\vv^s)) = c_{2|\nn|}(\KK_{\cc,\cc}(\vv^1,\vv^1,\dots,\vv^s,\vv^s)).}\]

By the Localization Theorem \ref{thm:localization}, we have an isomorphism
\[\textstyle{ \HH_{T_\star}^*(\MM_{\cc}(\vv^1,\dots,\vv^s) \times \MM_{\cc}(\uu^1,\dots,\uu^s)) \rightarrow \HH_{T_\star}^*(\MM_{\cc}(\vv^1,\dots, \vv^s) \times \MM_{\cc}(\uu^1,\dots,\uu^s)) = \bigoplus_{(\II,\JJ)} \HH_{T_\star}^*(\pt),}\]
where the direct sum runs over all $(\II,\JJ) \in \MM_{\cc}(\vv^1,\dots,\vv^s)^{T_\star} \times \MM_{\cc}(\uu^1,\dots,\uu^s)^{T_\star}$. Let $\xi \in \bigoplus_{(\II,\JJ)} \HH_{T_\star}^*(\pt)$ be the element with $(\II,\JJ)$-th component
\[\textstyle{ \xi_{(\II,\JJ)} = \frac{\epsilon}{e_{T_\star}(\T'_{(\II,\JJ)})},}\]
where $\T'_{(\II,\JJ)}$ is the tangent space of $(\II,\JJ)$ in $\MM_{\cc}(\vv^1,\dots,\vv^s) \times \MM_{\cc}(\uu^1,\dots,\uu^s)$. Let $\beta := j_*(\xi)$, where $j: \MM_{\cc}(\vv^1, \dots, \vv^s)^{T_\star} \times \MM_{\cc}(\uu^1,\dots,\uu^s)^{T_\star} \hookrightarrow \MM_{\cc}(\vv^1,\dots,\vv^s) \times \MM_{\cc}(\uu^1, \dots, \uu^s)$ is the natural inclusion. Denote by $j_{\II,\JJ}$ the inclusion $\{(\II,\JJ)\} \hookrightarrow \MM_{\cc}(\vv^1,\dots, \vv^s) \times \MM_{\dee}(\uu^1,\dots,\uu^s)$. Then, by \cite[Equation (9.3)]{cdks},
\[\textstyle{ j_{\II,\JJ}^*(\beta) = (j_{\II,\JJ}^* \circ j_\star)(\xi) = e_{T_\star}(\T'_{(\II,\JJ)}) \cup \xi_{(\II,\JJ)} = \epsilon.}\]
For $\ell=1,\dots, s$ and $k = 0, \dots, r_\ell$, we define
\begin{align*}
\textstyle{\frakE_k^\ell(\cc,\nn)} & \textstyle{:= -(R/r_\ell) \beta \cup c_{\tnv}(\KK_{\cc}(\nn; \ell, k)^-)} \\
& \textstyle{\in \HH_{T_\star}^{4|\nn|-2}\left(\coprod_{|\vv^i|=\nn_i} \MM_{\cc}(\vv^1,\dots,\vv^s) \times \MM_{\cc}(\vv^1,\dots,\vv^\ell - \1_k, \dots, \vv^s)\right),}
\end{align*}
\begin{align*}
\textstyle{\frakF_k^\ell(\cc, \nn)} & \textstyle{:= (R/r_\ell) \beta \cup c_{\tnv}(\KK_{\cc}(\nn; \ell, k)^+)} \\
& \textstyle{\in \HH_{T_\star}^{4|\nn| + 2}\left( \coprod_{|\vv^i|=\nn_i} \MM_{\cc}(\vv^1,\dots,\vv^s) \times \MM_{\cc}(\vv^1, \dots, \vv^\ell + \1_k, \dots, \vv^s)\right),}
\end{align*}
\begin{align*}
\textstyle{\frakH_k^\ell(\cc;\vv^1,\dots,\vv^s)} &\textstyle{:= \left( (\1_{\bar{c}})_k - \sum_{j=0}^{r_\ell-1} a^\ell_{kj} \vv^\ell_j\right) c_{\tnv}(\KK_{\cc,\cc}(\vv^1,\vv^1,\dots,\vv^s,\vv^s))} \\
&\textstyle{\in \HH_{T_\star}^{4(|\vv^1|+\cdots+|\vv^s|)}(\MM_{\cc}(\vv^1,\dots,\vv^s) \times \MM_{\cc}(\vv^1,\dots,\vv^s)),}
\end{align*}
where $a_{kj}^\ell$ is the $(k,j)$-th entry of the generalized Cartan matrix of type $\widehat{A}_{r_\ell-1}$. Note that, a priori, the elements $\frakE_k^\ell$, $\frakF_k^\ell$ and $\frakH_k^\ell$ do not define geometric operators since they do not lie in the middle degree cohomology of their respective quiver varieties. However, we can push these elements forward to the middle degree cohomology of the appropriate moduli spaces. For what follows, it will be convenient to view $\HH_{T_\star}^*(\MM_{\cc}(\vv^1,\dots,\vv^s) \times \MM_{\dee}(\uu^1,\dots,\uu^s))$ as a subset of $\HH_{T_\star}^*(\M_{\cc}(\nn)^{\Z_R} \times \M_{\dee}(\mm)^{\Z_R})$. That is, by identifying
\begin{align*}
\textstyle{\HH_{T_\star}^*(\M_{\cc}(\nn)^{\Z_R} \times \M_{\dee}(\mm)^{\Z_R})} &\textstyle{\cong \HH_{T_\star}^*\left( \coprod_{ {|\vv^i| = \nn_i} \atop {|\uu^i| = \mm_i}} \MM_{\cc}(\vv^1,\dots,\vv^s) \times \MM_{\dee}(\uu^1,\dots,\uu^s)\right)} \\
& \textstyle{\cong \bigoplus_{ {|\vv^i| = \nn_i} \atop {|\uu^i| = \mm_i}} \HH_{T_\star}^*(\MM_{\cc}(\vv^1,\dots,\vv^s) \times \MM_{\dee}(\uu^1,\dots,\uu^s)),}
\end{align*}
we will identify $\HH_{T_\star}^*(\MM_{\cc}(\vv^1,\dots,\vv^s) \times \MM_{\dee}(\uu^1,\dots,\uu^s))$ with its image under the inclusion
\[\textstyle{ \HH_{T_\star}^*\left(\MM_{\cc}(\vv^1,\dots,\vv^s) \times \MM_{\dee}(\uu^1,\dots,\uu^s)\right) \hookrightarrow \HH_{T_\star}^*(\M_{\cc}(\nn)^{\Z_R} \times \M_{\dee}(\mm)^{\Z_R}).}\]
In this way, we will consider
\[\textstyle{ \frakE_k^\ell(\cc,\nn) \in \HH_{T_\star}^*(\M_{\cc}(\nn)^{\Z_R} \times \M_{\cc}(\nn - \1_\ell)^{\Z_R}), \quad \frakF_k^\ell(\cc,\nn) \in \HH_{T_\star}^*(\M_{\cc}(\nn)^{\Z_R} \times \M_{\cc}(\nn + \1_\ell)^{\Z_R}),}\]
\[\textstyle{ \frakH_k^\ell(\cc; \vv^1,\dots,\vv^s) \in \HH_{T_\star}^*(\M_{\cc}(\nn)^{\Z_R} \times \M_{\cc}(\nn)^{\Z_R}).}\] 
Recall from the Localization Theorem \ref{thm:localization} that we have isomorphisms $f_1$ and $f_2$ as in the diagram below:

\begin{center}
\begin{tikzpicture}[>=stealth]

\draw (-0.9,0) node {$\HH_{T_\star}^*(\M_{\cc}(\nn) \times \M_{\dee}(\mm))$} (2.5,-1.5) node {$\HH_{T_\star}^*(\M_{\cc}(\nn)^{T_\star} \times \M_{\dee}(\mm)^{T_\star}).$} (5.3,0) node {$\HH_{T_\star}^*(\M_{\cc}(\nn)^{\Z_R} \times \M_{\dee}(\mm)^{\Z_R})$};

\draw[->] (0,-0.3) -- (1.2,-1.2); \draw[->] (5,-0.3) --  (3.8,-1.2); \draw[dashed,->] (2.7,0) -- (1.3,0);

\draw (0,-0.75) node {$f_1$} (5,-0.75) node {$f_2$};

\end{tikzpicture}
\end{center}
Let
\[\textstyle{ \mu: \HH_{T_\star}^*(\M_{\cc}(\nn)^{T_\star} \times \M_{\dee}(\mm)^{T_\star}) \rightarrow \HH_{T_\star}^*(\M_{\cc}(\nn)^{T_\star} \times \M_{\dee}(\mm)^{T_\star}), }\]
be the $\C(b_1,\dots,b_s,\epsilon)$-linear map determined by
\[\textstyle{ 1_{(\II,\JJ)} \mapsto \frac{e_{T_\star}(\T_{\II,\JJ}')}{e_{T_\star}(\T_{\II,\JJ})} 1_{(\II,\JJ)}, }\]
where $\T_{\II,\JJ}$ and $\T'_{\II,\JJ}$ are the tangent spaces of $(\II,\JJ)$ in $\M_{\cc}(\nn) \times \M_{\dee}(\mm)$ and $\M_{\cc}(\nn)^{\Z_R} \times \M_{\dee}(\mm)^{\Z_R}$, respectively. Let $\eta: \HH^*_{T_\star}(\M_{\cc}(\nn)^{\Z_R} \times \M_{\cc}(\mm)^{\Z_R}) \rightarrow \HH_{T_\star}^*(\M_{\cc}(\nn) \times \M_{\dee}(\mm))$ be the composition
\begin{equation}\textstyle{\label{eq:eta}
\eta := f_1^{-1} \circ \mu \circ f_2.
}\end{equation}

\begin{lemma}
The map $\eta$ is degree-preserving. In particular, the images of $\frakE_k^\ell(\cc,\nn)$, $\frakF_k^\ell(\cc,\nn)$ and $\frakH_k^\ell(\cc;\vv^1,\dots,\vv^s)$ under $\eta$ lie in the middle degree cohomology of $\M_{\cc}(\nn) \times \M_{\cc}(\nn \mp \1_\ell)$ and $\M_{\cc}(\nn) \times \M_{\cc}(\nn)$, respectively.
\end{lemma}

\begin{proof}
The map $f_2$ decreases the degree by $\rk(e_{T_\star}(\T'_{\II,\JJ}))$. The map $\mu$ increases the degree by $\rk(e_{T_\star}(\T'_{\II,\JJ})) - \rk(e_{T_\star}(\T_{\II,\JJ}))$. Finally, the map $f_1^{-1}$ increases the degree by $\rk(e_{T_\star}(\T_{\II,\JJ}))$. Thus, the composition of these maps is degree preserving.

The second statement in the lemma follows from the fact that $\frakE_k^\ell(\cc,\nn)$, $\frakF_k^\ell(\cc,\nn)$ and $\frakH_k^\ell$ are elements of degree $4|\nn| - 2$, $4|\nn|+2$ and $4|\nn|$, respectively.
\end{proof}

\begin{defn}[Geometric diagonal Chevalley operators]\label{defn:geoChevalley}
For $\ell=1,\dots,s$ and $k=0,\dots,r_\ell-1$, define operators
\[\textstyle{ \geoE_k^\ell, \geoF_k^\ell, \geoH_k^\ell: \bigoplus_{\nn,\cc} \HH_{T_\star}^{2|\nn|}(\M_{\cc}(\nn)) \rightarrow \bigoplus_{\nn,\cc} \HH_{T_\star}^{2|\nn|}(\M_{\cc}(\nn)),}\]
by
\[\textstyle{ \geoE_k^\ell|_{\HH_{T_\star}^{2|\nn|}(\M_{\cc}(\nn))} := \eta(\frakE_k^\ell(\cc,\nn))  \in \HH_{T_\star}^{4|\nn|-2}(\M_{\cc}(\nn) \times \M_{\cc}(\nn - \1_\ell)), }\]
\[\textstyle{ \geoF_k^\ell|_{\HH_{T_\star}^{2|\nn|}(\M_{\cc}(\nn))} := \eta(\frakF_k^\ell(\cc,\nn))  \in \HH_{T_\star}^{4|\nn|+2}(\M_{\cc}(\nn) \times \M_{\cc}(\nn + \1_\ell)), }\]
\[\textstyle{ \geoH^\ell_k|_{\HH_{T_\star}^{2|\nn|}(\M_{\cc}(\nn))} := \sum_{|\vv^i| = \nn_i} \eta( \frakH_k^\ell(\cc; \vv^1, \dots, \vv^s)) \in \HH_{T_\star}^{4|\nn|}(\M_{\cc}(\nn) \times \M_{\cc}(\nn)).}\]
The $\geoE^\ell_k$, $\geoF^\ell_k$ and $\geoH_k^\ell$ will be called \emph{geometric (diagonal) Chevalley operators}.
\end{defn}

Our next goal will be to describe the geometric Chevalley operators in terms of geometric Clifford operators. To that end, we prove the following useful lemma.

\begin{lemma}\label{lem:modOperator}
Let $(\II,\JJ) \in \M_{\cc}(\nn)^{T_\star} \times \M_{\dee}(\mm)^{T_\star}$ and $\beta \in \HH_{T_\star}^{2k}(\MM_{\cc}(\vv^1,\dots,\vv^s) \times \MM_{\dee}(\uu^1,\dots,\uu^s))$. Let
\[\textstyle{ \M_{\cc}(\nn) \times \M_{\dee}(\mm) \xhookleftarrow{i_1 \times i_2} \M_{\cc}(\nn)^{T_\star} \times \M_{\dee}(\mm)^{T_\star} \xhookrightarrow{j_1 \times j_2} \M_{\cc}(\nn)^{\Z_R} \times \M_{\dee}(\mm)^{\Z_R}, }\]
be the natural inclusions and let $(i_1 \times i_2)_{(\II,\JJ)}$ and $(j_1 \times j_2)_{(\II,\JJ)}$ denote the restrictions of $i_1 \times i_2$ and $j_1 \times j_2$, respectively, to $\{(\II,\JJ)\}$. If $(\II,\JJ) \in \MM_{\cc}(\vv^1,\dots,\vv^s)^{T_\star} \times \MM_{\dee}(\uu^1,\dots,\uu^s)^{T_\star},$ then
\[\textstyle{ \langle \eta( \beta \cup c_{|\nn|+|\mm|-k}(\KK_{\cc,\dee}(\vv^1,\uu^1,\dots,\vv^s,\uu^s)) [\II],[\JJ] \rangle = \langle \beta' \cup c_{|\nn|+|\mm|-k}(\K_{\cc,\dee}(\nn,\mm)) [\II], [\JJ] \rangle,}\]
where $\beta'$ is any preimage of $(j_1 \times j_2)_{(\II,\JJ)}^*(\beta)$ under the map $(i_1 \times i_2)^*_{(\II,\JJ)}$. Otherwise,
\[\textstyle{ \langle \eta( \beta \cup c_{|\nn|+|\mm|-k}(\KK_{\cc,\dee}(\vv^1,\uu^1,\dots,\vv^s,\uu^s))) [\II],[\JJ] \rangle = 0. }\]
\end{lemma}

\begin{proof}
To simplify notation, we write $\KK = \KK_{\cc,\dee}(\vv^1,\uu^1,\dots,\vv^s,\uu^s)$, $\K = \K_{\cc,\dee}(\nn,\mm)$ and $c = c_{|\nn|+|\mm|-k}$. We begin by explicitly computing $\eta(\beta \cup c(\KK))$ (recall $\eta = f_1^{-1} \circ \mu \circ f_2$ as in \eqref{eq:eta}). By definition of $f_2$,
\begin{align*}
\textstyle{f_2 (\beta \cup c(\KK))} &\textstyle{= \left( \frac{(j_1 \times j_2)_{(\bK,\bL)}^* (\beta \cup c(\KK))}{e_{T_\star}(\T'_{(\bK,\bL)})} \right)_{(\bK,\bL) \in \M_{\cc}(\nn)^{T_\star} \times \M_{\dee}(\mm)^{T_\star}}} \\
&\textstyle{= \left( \frac{\beta_{\bK,\bL} \cup (j_1 \times j_2)_{(\bK,\bL)}^* (c(\KK))}{e_{T_\star}(\T'_{(\bK,\bL)})} \right)_{(\bK,\bL) \in \M_{\cc}(\nn)^{T_\star} \times \M_{\dee}(\mm)^{T_\star}},} \\
\end{align*}
where $\T'_{(\bK,\bL)}$ is the tangent space of $(\bK,\bL)$ in $\M_{\cc}(\nn)^{\Z_R} \times \M_{\cc}(\mm)^{\Z_R}$, and $\beta_{\bK,\bL} = (j_1 \times j_2)_{(\bK,\bL)}^*(\beta)$. By applying $\mu$, we get
\[\textstyle{ \mu \circ f_2 (\beta \cup c(\KK)) = \left( \frac{ \beta_{\bK,\bL} \cup (j_1 \times j_2)_{(\bK,\bL)}^*(c(\KK))}{e_{T_\star}(\T_{(\bK,\bL)})} \right)_{(\bK,\bL) \in \M_{\cc}(\nn)^{T_\star} \times \M_{\dee}(\mm)^{T_\star}}, }\]
where $\T_{(\bK,\bL)}$ is the tangent space of $(\bK,\bL)$ in $\M_{\cc}(\nn) \times \M_{\dee}(\mm)$. By definition, the map $f_1^{-1} = (i_1 \times i_2)_*$, and thus,
\[\textstyle{ \eta(\beta \cup c(\KK)) = (i_1 \times i_2)_*\left( \frac{\beta_{\bK,\bL} \cup (j_1 \times j_2)_{(\bK,\bL)}^*(c(\KK))}{e_{T_\star}(\T_{(\bK,\bL)})} \right)_{(\bK,\bL) \in \M_{\cc}(\nn)^{T_\star} \times \M_{\dee}(\mm)^{T_\star}}. }\]
By definition of the bilinear form,
\begin{align*}\textstyle{ \langle \eta(\beta \cup c(\KK)) [\II], [\JJ] \rangle} &\textstyle{= (-1)^{|\mm|} p_* ((i_1 \times i_2)_*)^{-1} (\eta(\beta \cup c(\KK)) \cup [\II] \otimes [\JJ])} \\
&\textstyle{= (-1)^{|\mm|} p_* ((i_1 \times i_2)_*)^{-1} \left( \frac{\eta(\beta \cup c(\KK))}{e_{T_\star}(\T_{\II}^-) e_{T_\star}(\T_{\JJ}^-)} \cup (i_1 \times i_2)_*(1_{(\II,\JJ)}) \right),}
\end{align*}
where the last equality comes from the definition of $[\II]$ and $[\JJ]$. By the projection formula,
\[\textstyle{ \langle \eta(\beta \cup c(\KK)) [\II], [\JJ] \rangle = (-1)^{|\mm|} p_* \left( \frac{(i_1 \times i_2)^* (\eta(\beta \cup c(\KK)))}{e_{T_\star}(\T_{\II}^-) e_{T_\star}(\T_{\JJ}^-)} \cup 1_{(\II,\JJ)} \right).}\]
By \cite[Equation 9.3]{cdks}, $(i_1 \times i_2)^* (i_1 \times i_2)_*$ is simply multiplication by the Euler class of the tangent space. Hence,
\[\textstyle{ (i_1 \times i_2)^* \eta(\beta \cup c(\KK)) = \left( \beta_{\bK,\bL} \cup (j_1 \times j_2)_{(\bK,\bL)}^*( c(\KK) ) \right)_{(\bK,\bL) \in \M_{\cc}(\nn)^{T_\star} \times \M_{\dee}(\mm)^{T_\star}}.}\]
Thus,
\[\textstyle{ \langle\eta( \beta \cup c(\KK)) [\II], [\JJ] \rangle = (-1)^{|\mm|} p_* \left( \frac{ \beta_{\II,\JJ} \cup (j_1 \times j_2)_{(\II,\JJ)}^*(c(\KK)) }{e_{T_\star}(\T_{\II}^-) e_{T_\star}(\T_{\JJ}^-)} \right).}\]
By construction, if $(\II,\JJ) \notin \MM_{\cc}(\vv^1,\dots,\vv^s)^{T_\star} \times \MM_{\dee}(\uu^1,\dots,\uu^s)^{T_\star},$ then $(j_1 \times j_2)^*_{(\II,\JJ)} (c(\KK)) = 0$. On the other hand, if 
\[\textstyle{ (\II,\JJ) \in \MM_{\cc}(\vv^1,\dots,\vv^s)^{T_\star} \times \MM_{\dee}(\uu^1,\dots,\uu^s)^{T_\star}, }\]
then by functoriality of the Chern class and the construction of $\KK$,
\[\textstyle{ (j_1 \times j_2)^*_{(\II,\JJ)} (c(\KK)) = c(\KK_{(\II,\JJ)}) = c(\K_{(\II,\JJ)}).}\]
Therefore,
\[\textstyle{ \langle \eta( \beta \cup c(\KK)) [\II], [\JJ] \rangle = (-1)^{|m|} \frac{\beta_{\II,\JJ} \cup c(\K_{(\II,\JJ)}) } {e_{T_\star}(\T_{\II}^-) e_{T_\star}(\T_{\JJ}^-)}  = \frac{\beta_{\II,\JJ} \cup c(\K_{(\II,\JJ)}) } {e_{T_\star}(\T_{\II}^-) e_{T_\star}(\T_{\JJ}^+)} = \langle \beta' \cup c(\K) [\II],[\JJ] \rangle,}\]
where the last equality follows from Lemma \ref{lem:operator}.
\end{proof}

\begin{cor}\label{cor:rest}
Let $(\II,\JJ) \in \M_{\cc}(\nn)^{T_\star} \times \M_{\cc}(\mm)^{T_\star}$.

\begin{enumerate}

\item For $\mm = \nn - \1_\ell$, if $(\II,\JJ) \in \MM_{\cc}(\vv^1,\dots,\vv^s)^{T_\star} \times \MM_{\cc}(\vv^1,\dots,\vv^\ell - \1_k, \dots,\vv^s)^{T_\star},$ then
\[\textstyle{ \langle \geoE_k^\ell [\II], [\JJ] \rangle = \langle \geoP_\ell(1) [\II],[\JJ] \rangle.}\]
Otherwise
\[\textstyle{ \langle \geoE_k^\ell [\II], [\JJ] \rangle = 0.}\]

\item For $\mm = \nn + \1_\ell$, if $(\II,\JJ) \in \MM_{\cc}(\vv^1,\dots,\vv^s)^{T_\star} \times \MM_{\cc}(\vv^1,\dots,\vv^\ell + \1_k, \dots,\vv^s)^{T_\star},$ then
\[\textstyle{ \langle \geoF_k^\ell [\II], [\JJ] \rangle = \langle \geoP_\ell(-1) [\II],[\JJ] \rangle.}\]
Otherwise
\[\textstyle{ \langle \geoF_k^\ell [\II], [\JJ] \rangle = 0.}\]

\item For $\mm = \nn$,
\[\textstyle{ \langle \geoH_k^\ell [\II], [\JJ] \rangle = \left( (\1_{\bar{\cc}_\ell})_k - \sum_j a^\ell_{kj} \vv^\ell_j \right) \delta_{\II, \JJ}.}\] 

\end{enumerate}

\end{cor}

\begin{proof}
Statements (1) and (2) follow directly from Lemma \ref{lem:modOperator} and the definitions of $\geoP_\ell(\pm 1)$. The third statement follows from Lemma \ref{lem:modOperator} and the fact that $c_{\tnv}(\K_{\cc,\cc}(\nn,\nn)) = \id$ as an operator on $\HH_{T_\star}^{2|\nn|}(\M_{\cc}(\nn))$ (see Definition \ref{defn:geoHeisenberg}).
\end{proof}

\begin{prop}\label{prop:P(-1)}
Let $(\II,\JJ) \in \M_{\cc}(\nn)^{T_\star} \times \M_{\cc}(\nn+\1_\ell)^{T_\star}$. If $\lambda(\II_\alpha) = \lambda(\JJ_\alpha)$ for all $\alpha \ne \ell$ and $\lambda(\JJ_\ell)$ can be obtained by adding one box to $\lambda(\II_\ell)$, then
\[\textstyle{ \langle \geoP_\ell(-1) [\II], [\JJ] \rangle = 1. }\]
Otherwise, $\langle \geoP_\ell(-1) [\II], [\JJ] \rangle = 0.$
\end{prop}

\begin{proof}
This is a direct result of \cite[Proposition 5.3 and the proof of Theorem 3.14]{savlic}. Note that in our case, $\lambda(\JJ_\ell) - \lambda(\II_\ell)$ consists of a single box.
\end{proof}

\begin{lemma}\label{lem:boxCount}
If $\II \in \MM_{\cc}(\vv^1, \dots, \vv^s)^{T_\star}$, then $\vv^\ell_k$ is equal to the number of boxes in $\lambda(\II_\ell)$ whose residue is congruent to $k-\cc_\ell$ modulo $r_\ell$.
\end{lemma}

\begin{proof}
Let $\II_\ell = [A,B,i] \in \M_{\cc_\ell}(1,\nn_\ell)^{T_\star}$. By \cite[Proposition 2.9]{nak3}, $\lambda(\II_\ell)$ is obtained from $\II_\ell$ by drawing a box in the $(p,q)$-th position if $A^{p-1} B^{q-1} i \ne 0$ (note that our Young diagrams are rotated $90 \degree$ clockwise from those in \cite{nak3}). From Equation \eqref{eq:Vk},
\[\textstyle{\vv^\ell_k = \dim \vspan\{ A^p B^q i \; | \; q-p \equiv k - \cc_\ell \!\!\! \mod r_\ell \}.}\]
Since the nonzero $A^p B^q i$ are linearly independent, the boxes $(p,q) \in \lambda(\II_\ell)$ whose residue is congruent to $k-\cc_\ell \!\!\! \mod r_\ell$ are in one-to-one correspondence with a basis of $\vspan\{ A^p B^q i \; | \; q-p \equiv k - \cc_\ell \!\!\! \mod r_\ell \}$. Thus, $\vv^\ell_k$ is equal to the number of such boxes.
\end{proof}

\begin{lemma}\label{lem:geoF}
Let $(\II,\JJ) \in \M_{\cc}(\nn)^{T_\star} \times \M_{\cc}(\nn + \1_\ell)^{T_\star}$. Then
\[\textstyle{ \langle \geoF_k^\ell [\II], [\JJ] \rangle = 1, }\]
if $\lambda(\II_\alpha) = \lambda(\JJ_\alpha)$ for all $\alpha \ne \ell$ and $\lambda(\JJ_\ell)$ can be obtained from $\lambda(\II_\ell)$ by adding one box whose residue is congruent to $k - \cc_\ell$ modulo $r_\ell$. Otherwise, $\langle \geoF_k^\ell [\II], [\JJ] \rangle = 0$.
\end{lemma}

\begin{proof}
This follows immediately from Corollary \ref{cor:rest}, Proposition \ref{prop:P(-1)} and Lemma \ref{lem:boxCount}.
\end{proof}

\begin{prop}\label{prop:psi}
Let $(\II,\JJ) \in \M_{\cc}(\nn)^{T_\star} \times \M_{\cc}(\nn+1_\ell)^{T_\star}$. Then for all $i \in \Z$,
\[\textstyle{ \langle \geoPsi_\ell(i) \geoPsi_\ell^*(i-1) [\II], [\JJ] \rangle = \begin{cases} 1, & \text{if } \JJ = \psi_\ell(i)\psi_\ell^*(i-1) \II, \\ 0, & \text{otherwise.} \end{cases} }\]
\end{prop}

\begin{proof}
This follows immediately from Theorem \ref{thm:geoClifford}.
\end{proof}

\begin{thm}\label{thm:geoE,geoF}
For $\ell=1,\dots,s$ and $k=0,1,\dots,r_\ell-1$,
\[\textstyle{ \geoE^\ell_k = \sum_{i \in \Z} \geoPsi_\ell(k + ir_\ell) \geoPsi^*_\ell(k+ir_\ell+1), \qquad \geoF^\ell_k = \sum_{i \in \Z} \geoPsi_\ell(k + i r_\ell + 1) \geoPsi^*_\ell(k + i r_\ell),}\]
as operators on $\bigoplus_{\nn,\cc} \HH_{T_\star}^{2|\nn|}(\M_{\cc}(\nn))$.
\end{thm}

\begin{proof}
Let $\II, \JJ$ be two $s$-tuples of semi-infinite monomials of charge $\cc$. We first prove that
\begin{equation}\textstyle{\label{equn:geoF}
\langle \geoF_k^\ell [\II], [\JJ] \rangle = \left\langle \sum_{i \in \Z} \geoPsi_\ell(k + i r_\ell + 1) \geoPsi^*_\ell(k + i r_\ell) [\II],[\JJ] \right\rangle.
}\end{equation}
If there exists an $\alpha \ne \ell$ such that $\II_\alpha \ne \JJ_\alpha$, then both sides of Equation \eqref{equn:geoF} are zero and we are done. Thus, we assume that $\II_\alpha = \JJ_\alpha$ for all $\alpha \ne \ell$. Write
\[\textstyle{ \II_\ell = i_1 \wedge i_2 \wedge i_3 \wedge \cdots, \quad \text{and} \quad \JJ_\ell = j_1 \wedge j_2 \wedge j_3 \wedge \cdots,}\]
where $i_m, j_m \in \Z$. Recall that the number of boxes in the $m$-th row of $\lambda(\II_\ell)$ is $i_m - \cc_\ell + m - 1$ (likewise for $\lambda(\JJ_\ell)$). Suppose that $\langle \geoF_k^\ell [\II], [\JJ] \rangle = 1$. Then by Lemma \ref{lem:geoF}, $\lambda(\JJ_\ell)$ is obtained by adding one box of the appropriate residue to $\lambda(\II_\ell)$. This means that there exists an $m \in \N$ such that $j_n = i_n$ for all $n \ne m$ and $j_m = i_m + 1$. The position of the added box is thus $(m, j_m - \cc_\ell + m - 1)$. The residue of the added box must be congruent to $k - \cc_\ell$ modulo $r_\ell$, and so,
\[\textstyle{ ( j_m - \cc_\ell + m - 1 ) - m = j_m - \cc_\ell - 1 \equiv k - \cc_\ell \!\!\! \mod r_\ell, }\]
or equivalently
\[\textstyle{ j_m - k - 1 \equiv 0 \!\!\! \mod r_\ell.}\]
Thus, $j_m = k + ir_\ell + 1$, for some $i \in \Z$. Now, since $j_n = i_n$ for all $n \ne m$ and $j_m = i_m + 1$, we have that
\[\textstyle{ \JJ = \psi_\ell(j_m) \psi_\ell^*(j_m -1) \II, \quad \text{and} \quad \JJ \ne \psi_\ell(j) \psi_\ell^*(j-1) \II, \quad \text{for all } j \ne j_m.}\]
Hence,
\[\textstyle{ \left\langle \sum_{i \in \Z}\geoPsi_\ell(k + i r_\ell + 1) \geoPsi_\ell^*(k + i r_\ell) [\II],[\JJ] \right\rangle = \langle \geoPsi_\ell(j_m) \geoPsi_\ell^*(j_m-1) [\II],[\JJ] \rangle = 1,}\]
by Proposition \ref{prop:psi}. We have therefore proven that Equation \eqref{equn:geoF} holds when $\langle \geoF_k^\ell [\II], [\JJ] \rangle = 1$.

Suppose now that $\langle \geoF_k^\ell [\II], [\JJ] \rangle = 0$. Then $\lambda(\JJ_\ell)$ cannot be obtained by adding one box to $\lambda(\II_\ell)$. In particular, $\JJ \ne \psi_\ell(i) \psi_\ell^*(i-1) \II$ for all $i \in \Z$. Thus,
\[\textstyle{ \left\langle \sum_{i \in \Z} \geoPsi_\ell(k + i r_\ell + 1) \geoPsi_\ell^*(k + i r_\ell) [\II],[\JJ] \right\rangle = 0.}\]
Therefore, Equation \eqref{equn:geoF} holds in all cases, and so
\[\textstyle{ \geoF^\ell_k = \sum_{i \in \Z} \geoPsi_\ell(k + i r_\ell + 1) \geoPsi^*_\ell(k + i r_\ell).}\]

Now, since $\geoP_\ell(-1)$ and $\geoP_\ell(1)$ are adjoint (see \cite[Lemma 3.13]{savlic}), it follows from Corollary \ref{cor:rest} that $\geoE^\ell_k$ and $\geoF_k^\ell$ are adjoint. Moreover, since $\geoPsi_\ell(i)$ and $\geoPsi_\ell^*(i)$ are adjoint (see \cite[Lemma 3.5]{savlic}), it follows that
\[\textstyle{ \sum_{i \in \Z} \geoPsi_\ell(k + i r_\ell + 1) \geoPsi_\ell^*(k + i r_\ell), \quad \text{and} \quad \sum_{i \in \Z} \geoPsi_\ell(k + i r_\ell) \geoPsi_\ell^*(k + i r_\ell + 1),}\]
are adjoint. Therefore, for all $s$-tuples of semi-infinite monomials $\II$, $\JJ$,
\begin{align*}
\textstyle{\langle \geoE^\ell_k [\II], [\JJ] \rangle} &\textstyle{= \langle [\II], \geoF_k^\ell [\JJ] \rangle = \left\langle [\II], \sum_{i \in \Z} \geoPsi_\ell(k + i r_\ell + 1) \geoPsi_\ell^*(k + i r_\ell) [\JJ] \right\rangle} \\
&\textstyle{= \left\langle \sum_{i \in \Z} \geoPsi_\ell(k + i r_\ell) \geoPsi_\ell^*(k + i r_\ell +1) [\II],[\JJ] \right\rangle.}
\end{align*}
Thus,
\[\textstyle{ \geoE^\ell_k = \sum_{i \in \Z} \geoPsi_\ell(k + i r_\ell) \geoPsi^*_\ell(k + i r_\ell + 1).\qedhere}\]
\end{proof}

\begin{prop}\label{prop:[E,F]=H}
For all $\ell=1,\dots,s$ and $k=0,1,\dots,r_\ell-1$,
\[\textstyle{ [\geoE_k^\ell,\geoF_k^\ell] = \geoH_k^\ell.}\]
\end{prop}

\begin{proof}
If $r_\ell = 1$, then
\[\textstyle{ \geoE_0^\ell = \geoP_\ell(1), \quad \geoF_0^\ell = \geoP_\ell(-1), \quad \geoH_0^\ell = \id,}\]
and the result follows from Theorem \ref{thm:geoHeisenberg}.

Now, fix an $\ell \in \{1,\dots,s\}$ such that $r_\ell \ge 2$. For any semi-infinite monomial $I$ of charge $c \in \Z$, we will say that a box $(p,q) \in \lambda(I)$ is a $k$\emph{-box} if its residue is congruent to $k-c \!\!\! \mod r_\ell$. We will say that a box $(p,q) \in \lambda(I)$ is $k$\emph{-removable} if $(p,q)$ is a $k$-box and $(p,q)$ may be removed from $\lambda(I)$ to create a new Young diagram. Denote the set of $k$-removable boxes of $\lambda(I)$ by $R$. We will say that a box $(p,q) \notin \lambda(I)$ is $k$\emph{-addable} if $(p,q)$ is a $k$-box and $(p,q)$ may be added to $\lambda(I)$ to create a new Young diagram. Denote the set of $k$-addable boxes of $\lambda(I)$ by $A$.

By Theorem \ref{thm:geoE,geoF}, we have that, for all $\II \in \MM_{\cc}(\vv^1,\dots,\vv^s)^{T_\star}$
\[\textstyle{ \geoE_k^\ell [\II] = \sum_{\JJ} [\JJ], \quad \text{and} \quad \geoF_k^\ell [\II] = \sum_{\bK} [\bK],}\]
where the $\JJ$ run over all semi-infinite monomials such that $\JJ_i = \II_i$ for all $i \ne \ell$ and $\JJ_\ell$ is obtained from $\II_\ell$ by removing a $k$-removable box, and the $\bK$ run over all semi-infinite monomials such that $\bK_i = \II_i$ for all $i \ne \ell$ and $\bK_\ell$ is obtained from $\II_\ell$ by adding a $k$-addable box. Thus, it is easy to see that
\[\textstyle{ [\geoE_k^\ell,\geoF_k^\ell] [\II] = (|A| - |R|) [\II],}\]
where here $A$ and $R$ refer to the sets of $k$-addable and $k$-removable boxes of $\II_\ell$, respectively. Therefore, it suffices to show that
\[\textstyle{ |A| - |R| = (\1_{\bar{c}})_k - \sum_{j=0}^{r_\ell} a^\ell_{kj} \vv^\ell_j = \delta_{\bar{c},\bar{k}} - 2\vv^\ell_k + v^\ell_{k+1} + \vv^\ell_{k-1} = \delta_{\bar{c},\bar{k}} +  (\vv^\ell_{k+1} - \vv^\ell_k) + (\vv^\ell_{k-1} - \vv^\ell_k),}\]
where, of course, the indices of $\vv^\ell$ are taken modulo $r_\ell$.

For the remainder of the proof, we will identify $\II_\ell$ with its Young diagram $\lambda = \lambda(\II_\ell)$. The case where $\lambda$ is the empty Young diagram is trivial, so we assume $\lambda$ consists of at least one box. The $k$-\emph{border} of $\lambda$ is the set
\[\textstyle{ B := \{ k\text{-boxes } (p,q) \in \lambda \; | \; (p+1,q), (p,q+1), \text{ or } (p+1,q+1) \notin \lambda \}.}\]
We partition $B$ with respect to the main diagonal of $\lambda$ as follows:
\[\textstyle{ B = U \; \dot\cup \; M \; \dot\cup \; L,}\]
\[\textstyle{ U = \{ (p,q) \in B \; | \; q > p \}, \quad M = \{ (p,p) \in B \}, \quad L = \{ (p,q) \in B \; | \; p > q\}.}\]
The sets $B, U, M$ and $L$ are illustrated in the following diagram:
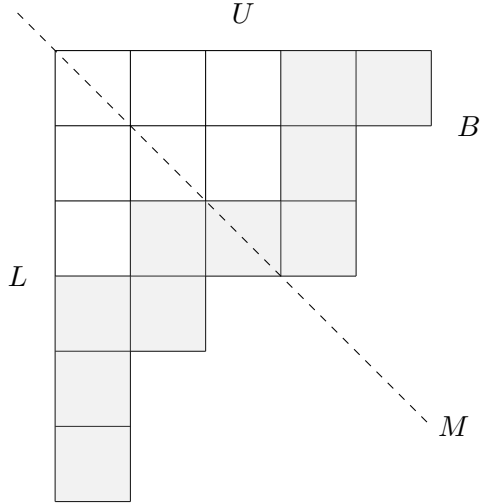
\begin{center}
\begin{tikzpicture}

\draw (0,0) -- (0,-6); \draw (1,0) -- (1,-6); \draw (2,0) -- (2,-4); \draw (3,0) -- (3,-3); \draw (4,0) -- (4,-3); \draw (5,0) -- (5,-1);

\draw (0,0) -- (5,0); \draw (0,-1) -- (5,-1); \draw (0,-2) -- (4,-2); \draw (0,-3) -- (4,-3); \draw (0,-4) -- (2,-4); \draw (0,-5) -- (1,-5); \draw (0,-6) -- (1,-6);

\fill[gray!40,nearly transparent] (3,0) -- (5,0) -- (5,-1) -- (4,-1) -- (4,-3) -- (2,-3) -- (2,-4) -- (1,-4) -- (1,-6) -- (0,-6) -- (0,-3) -- (1,-3) -- (1,-2) -- (3,-2) -- cycle;

\draw[dashed] (-0.5,0.5) -- (5,-5);

\draw (-0.5,-3) node {$L$} (2.5,0.5) node {$U$} (5.5,-1) node {$B$} (5.3,-5) node {$M$};

\end{tikzpicture}
\end{center}
Here $B$ is the set of $k$-boxes in the shaded region, $U$ (resp.\ $L$) is the set of boxes in $B$ that lie above (resp.\ below) the dashed line, and $M$ is the intersection of $B$ with the dashed line. Let $\Delta_{\rightarrow}$ (resp.\ $\Delta_{\downarrow}$) be the subset of $B$ consisting of boxes which have a right (resp.\ lower) neighbour. That is,
\[\textstyle{ \Delta_\rightarrow = \{ (p,q) \in B \; | \; (p,q+1) \in \lambda\}, \quad \text{and} \quad \Delta_\downarrow = \{ (p,q) \in B \; | \; (p+1,q) \in \lambda\}.}\]
Denote by $\Delta'_\rightarrow$ and $\Delta'_\downarrow$ the complements of $\Delta_\rightarrow$ and $\Delta_\downarrow$, respectively, in $B$.

We begin with the case where $k \not\equiv \cc_\ell \!\!\! \mod r_\ell$ (so that $M = \emptyset$). In the portion of $\lambda$ above the main diagonal, every $(k+1)$-box occurs as the right neighbour of a $k$-box. Conversely, every right neighbour of a $k$-box (if it has one) is a $(k+1)$-box. If a $k$-box does not have a right neighbour, it must therefore lie in $B$. Hence, the number of $k$-boxes less the number of $(k+1)$-boxes above the main diagonal is equal to $|U \cap \Delta'_\rightarrow|$. In the portion of $\lambda$ below and including the main diagonal, we can dualize this argument and conclude that the number of $(k+1)$-boxes less the number of $k$-boxes is equal to the number of $(k+1)$-boxes (on the border) with no lower neighbour. Any $(k+1)$-box not in the first column and without a lower neighbour has a left neighbour, which must be a $k$-box and lie in $U$ and $\Delta_\rightarrow$. Hence, the number of $(k+1)$-boxes less the number of $k$-boxes is equal to $|L \cap \Delta_\rightarrow| + \delta$, where $\delta = 1$ if the last box of the first column of $\lambda$ is a $(k+1)$-box and $\delta=0$ otherwise. Since $\vv^\ell_k$ (resp.\ $\vv^\ell_{k+1}$) is the number of $k$-boxes (resp.\ $(k+1)$-boxes) in $\lambda$,
\begin{equation}\textstyle{
\vv^\ell_{k+1} - \vv^\ell_k = |L \cap \Delta_\rightarrow| + \delta - |U \cap \Delta'_\rightarrow|.
}\end{equation}
By a completely analogous argument,
\begin{equation*}
\vv^\ell_{k-1} - \vv^\ell_k = |U \cap \Delta_\downarrow| + \delta' - |L \cap \Delta'_\downarrow|,
\end{equation*}
where $\delta'=1$ if the last box of the first row of $\lambda$ is a $(k-1)$-box and $\delta'=0$ otherwise. Now, we note that a $k$-box is $k$-removable if and only if it has no right or lower neighbours. Hence, $R = \Delta'_\rightarrow \cap \Delta'_\downarrow$ and so
\begin{equation}\textstyle{\label{eq:R}
|R| = |\Delta'_\rightarrow \cap \Delta'_\downarrow| = |B| - |\Delta_\rightarrow \cup \Delta_\downarrow|.}\end{equation}
We can add a $k$-box at the end of the first row (resp.\ column) if and only if the last box of the first row (resp.\ column) is a $(k-1)$-box (resp.\ $(k+1)$-box). A $k$-box $(p,q)$, where $p,q \ne 1$, is $k$-addable if and only if both $(p-1,q)$ and $(p,q-1)$ are in $\lambda$, which occurs if and only if $(p-1,q-1) \in \Delta_\rightarrow \cap \Delta_\downarrow$. Hence,
\begin{equation}\textstyle{\label{eq:Add}
|A| = |\Delta_\rightarrow \cap \Delta_\downarrow| + \delta + \delta' = |\Delta_\rightarrow| + |\Delta_\downarrow| - |\Delta_\rightarrow \cup \Delta_\downarrow| + \delta + \delta'.
}\end{equation}
Since $M = \emptyset$, we have that $B = U \; \dot\cup \; L$, and so
\begin{equation}\textstyle{\label{eq:Delta}
|\Delta_{\rightarrow/\downarrow}| = |U \cap \Delta_{\rightarrow/\downarrow}| + |L \cap \Delta_{\rightarrow/\downarrow}|
}\end{equation}
Combining equations \eqref{eq:R}, \eqref{eq:Add} and \eqref{eq:Delta}, we get
\begin{align*}
\textstyle{|A| - |R|} &\textstyle{= |U \cap \Delta_\rightarrow| + |L \cap \Delta_\rightarrow| + |U \cap \Delta_\downarrow| + |L \cap \Delta_\downarrow| - |B| + \delta + \delta'} \\
&\textstyle{= |U| - |U \cap \Delta'_\rightarrow| + |L \cap \Delta_\rightarrow| + |U \cap \Delta_\downarrow| + |L| - |L \cap \Delta'_\downarrow| - |B| + \delta + \delta'} \\
&\textstyle{= - |U \cap \Delta'_\rightarrow| + |L \cap \Delta_\rightarrow| + |U \cap \Delta_\downarrow| - |L \cap \Delta'_\downarrow| + \delta + \delta'} \\
&\textstyle{= (\vv^\ell_{k+1} - \vv^\ell_k) + (\vv^\ell_{k-1} - \vv^\ell_k).}
\end{align*}

In the case where $k \equiv \cc_\ell \!\!\! \mod r_\ell$, one only has to make a few modifications to the above arguments owing to the fact that $M$ now consists of one box. In the portion of $\lambda$ above and including the main diagonal, the number of $k$-boxes less the number of $(k+1)$-boxes is $|U \cap \Delta'_\rightarrow| + |M \cap \Delta'_\rightarrow|$. In the portion of $\lambda$ below the main diagonal, one again has that the number of $(k+1)$-boxes less the number of $k$-boxes is $|L \cap \Delta_\rightarrow| + \delta$. Hence,
\begin{equation*}
\vv^\ell_{k+1} - \vv^\ell_k = |L \cap \Delta_\rightarrow| + \delta - |U \cap \Delta'_\rightarrow| - |M \cap \Delta'_\rightarrow|.
\end{equation*}
And again by a completely analogous argument,
\begin{equation}\textstyle{
\vv^\ell_{k-1} - \vv^\ell_k = |U \cap \Delta_\downarrow| + \delta' - |L \cap \Delta'_\downarrow| - |M \cap \Delta'_\downarrow|.
}\end{equation}
One can compute $|A|$ and $|R|$ exactly as in equations \eqref{eq:Add} and \eqref{eq:R}. The difference now is that, since $M \ne \emptyset$, we have $B = U \; \dot\cup \; M \; \dot\cup \; L$, and so
\[\textstyle{ |\Delta_{\rightarrow/\downarrow}| = |U \cap \Delta_{\rightarrow/\downarrow}| + |M \cap \Delta_{\rightarrow/\downarrow}| + |L \cap \Delta_{\rightarrow/\downarrow}|.}\]
Therefore, using similar calculations as before,
\begin{align*}
\textstyle{|A| - |R|} &\textstyle{= |U| - |U \cap \Delta'_\rightarrow| + |L \cap \Delta_\rightarrow| + |U \cap \Delta_\downarrow| + |L| - |L \cap \Delta'_\downarrow| - |B| + \delta + \delta'} \\
& \textstyle{\qquad + |M \cap \Delta_\rightarrow| + |M \cap \Delta_\downarrow|\textstyle} \\
&\textstyle{= |U| - |U \cap \Delta'_\rightarrow| + |L \cap \Delta_\rightarrow| + |U \cap \Delta_\downarrow| |L| - |L \cap \Delta'_\downarrow| - |B| + \delta + \delta'} \\
& \textstyle{\qquad + |M| - |M \cap \Delta'_\rightarrow| + |M| - |M \cap \Delta'_\downarrow|} \\
&\textstyle{= - |U \cap \Delta'_\rightarrow| + |L \cap \Delta_\rightarrow| + |U \cap \Delta_\downarrow| - |L \cap \Delta'_\downarrow| + \delta + \delta' - |M \cap \Delta'_\rightarrow| + |M| - |M \cap \Delta'_\downarrow|} \\
&\textstyle{= 1 + (\vv^\ell_{k+1} - \vv^\ell_k) + (\vv^\ell_{k-1} - \vv^\ell_k),}
\end{align*}
which completes the proof.
\end{proof}

By Proposition \ref{prop:[E,F]=H} and Theorem \ref{thm:geoE,geoF} (as compared to Lemma \ref{lem:glaction}), the operators $\geoE_k^\ell$, $\geoF_k^\ell$ and $\geoH_k^\ell$ give a geometric realization of the action of $\widehat\sla_{r_\ell}$ on $\bigoplus_{\nn,\cc} \HH_{T_\star}^*(\M_{\cc}(\nn))$. To complete our geometric realization of $V_\text{basic}(\widehat\gl_r)$, it remains to construct the action of the remaining Chevalley generators of $\widehat\sla_r$ (the generators lying in the off-diagonal blocks) and the loops on the identity.

Let $\cc \in \Z^s$ and $i \ne j \in \{1,\dots,s\}$. Suppose $\nn, \mm \in \N^s$ such that $\nn_\ell = \mm_\ell$ for all $\ell \ne i,j$. By equations \eqref{eq:HeisenChern} and \eqref{eq:CliffChern}, we have that
\[\textstyle{ c_{\tnv}(\K_{\cc, \cc + \1_i - \1_j} (\nn,\mm)) = c_{|\nn| + |\mm|} ( \K_{\cc, \cc + \1_i - \1_j} (\nn,\mm)). }\]
For $i = 1, \dots, s-1$ and $k \in \Z$, define
\[\textstyle{ \K^+_{\cc}(\nn;i,k) := \K_{\cc,\cc + \1_i - \1_{i+1}} (\nn, \nn + ( (k+1)r_i - \cc_i - 1) \1_i - (k r_{i+1} - \cc_{i+1} + 1) \1_{i+1}), }\]
\[\textstyle{ \K^-_{\cc}(\nn;i,k) := \K_{\cc,\cc + \1_{i+1} - \1_i} (\nn, \nn + ( (k-1)r_{i+1} - \cc_{i+1}) \1_{i+1} - (k r_i - \cc_i) \1_i). }\]
We also define
\[\textstyle{ \K^+_{\cc}(\nn;0,k) := \K_{\cc,\cc + \1_s - \1_1} (\nn, \nn + ( k r_s - \cc_s - 1) \1_s - ( k r_1 - \cc_1 + 1) \1_1), }\]
\[\textstyle{ \K^-_{\cc}(\nn;0,k) := \K_{\cc,\cc + \1_1 - \1_s} (\nn, \nn + ( k r_1 - \cc_1) \1_1 - ( k r_s - \cc_s) \1_s).}\]

\begin{defn}[Geometric off-diagonal Chevalley operators]
For $i=0,1,\dots,s-1$, define operators
\[\textstyle{ \geoE'_i, \geoF'_i, \geoH_i' : \bigoplus_{\nn,\cc} \HH_{T_\star}^{2|\nn|}(\M_{\cc}(\nn)) \to \bigoplus_{\nn,\cc} \HH_{T_\star}^{2|\nn|}(\M_{\cc}(\nn)), }\]
by
\[\textstyle{ \geoE'_i|_{\HH_{T_\star}^{2|\nn|}(\M_{\cc}(\nn))} = (-1)^{\cc_i} \sum_{k \in \Z} c_{\tnv}(\K_{\cc}^+(\nn;i,k)), \quad \geoF'_i|_{\HH_{T_\star}^{2|\nn|}(\M_{\cc}(\nn))} = (-1)^{\cc_i} \sum_{k \in \Z} c_{\tnv}(\K_{\cc}^-(\nn;i,k)), }\]
for $i=1,\dots,s-1$, and
\[\textstyle{ \geoE'_0|_{\HH_{T_\star}^{2|\nn|}(\M_{\cc}(\nn))} = (-1)^{\cc_1 + \dots + \cc_{s-1}} \sum_{k \in \Z} c_{\tnv}(\K_{\cc}^+(\nn;0,k)),}\]
\[\textstyle{ \geoF'_0|_{\HH_{T_\star}^{2|\nn|}(\M_{\cc}(\nn))} = (-1)^{\cc_1 + \dots + \cc_{s-1}} \sum_{k \in \Z} c_{\tnv}(\K_{\cc}^-(\nn;0,k)), }\]
and finally, $\geoH'_i := [\geoE'_i, \geoF'_i]$, for all $i=0,\dots,s-1$.
\end{defn}

\begin{rmk}
It will follow from Theorem \ref{thm:main}, that for $\II \in \MM_{\cc}(\vv^1,\dots,\vv^s)^{T_\star}$, the element $[\II]$ lies in the $\vv^\ell$-weight space of $\widehat\gl_{r_\ell}$, thus allowing us to define the operators $\geoH_k^\ell$ explicitly in terms of the $\vv^\ell$. For the off-diagonal Chevalley operators, the $\vv^\ell$ no longer encode information about the weights of $\widehat\gl_r$. Hence, our defining $\geoH'_i$ simply as the commutator of $\geoE'_i$ and $\geoF_i'$ seems the best we can achieve.
\end{rmk}

\begin{lemma}\label{lem:offdiag}
As operators on $\A$,
\[\textstyle{ \geoE'_i = \sum_{k \in \Z} \geoPsi_i((k+1)r_i) \geoPsi_{i+1}^*(k r_{i+1} + 1), \quad \geoF'_i = \sum_{k \in \Z} \geoPsi_{i+1}((k-1) r_{i+1} + 1) \geoPsi^*_i(k r_i),}\]
for all $i=1, \dots, s-1$, and
\[\textstyle{ \geoE'_0 = \sum_{k \in \Z} \geoPsi_s(k r_s) \geoPsi^*_1(k r_1+1), \quad \geoF'_0 = \sum_{k \in \Z} \geoPsi_1 (kr_1+1) \geoPsi_s^*(k r_s).}\]
\end{lemma}

\begin{proof}
Fix $i \in \{ 1, \dots, s-1 \}$. Consider the restriction of $\geoE'_i$ to $\HH_{T_\star}^{2|\nn|}(\M_{\cc}(\nn))$. To simplify notation, let $\dee = \cc + \1_i - \1_{i+1}$ and $\mm^k = \nn + ((k+1)r_i - \cc_i -1)\1_i - (k r_{i+1} - \cc_{i+1} + 1) \1_{i+1}$. Then
\begin{align*}
\textstyle{\geoE_i'} &\textstyle{= (-1)^{\cc_i} \sum_{k \in \Z} c_{\tnv}(\K_{\cc,\dd}(\nn,\mm^k)) = (-1)^{\cc_i} \sum_{k \in \Z} c_{\tnv}\left( \bigoplus_{\ell=1}^s f_\ell^*\K_{\cc_\ell,\dee_\ell} (1,\nn_\ell, (\mm^k)_\ell) \right)} \\
&\textstyle{= (-1)^{\cc_i} \sum_{k \in \Z} \left(\prod_{\ell=1}^s ( 1^{\otimes \ell - 1} \otimes c_{\tnv}(\K_{\cc_\ell, \dee_\ell}(1,\nn_\ell,(\mm^k)_\ell)) \otimes 1^{\otimes s - \ell}\right),}
\end{align*}
where the second equality follows from Lemma \ref{lem:vbpullback} and the third equality is the application of the K\"unneth formula. For $\ell \ne i, i+1$,
\[\textstyle{ c_{\tnv}(\K_{\cc_\ell,\dee_\ell} (1,\nn_\ell, (\mm^k)_\ell)) = c_{\tnv}(\K_{\cc_\ell,\cc_\ell}(1,\nn_\ell, \nn_\ell)) = 1,}\]
by \cite[Lemma 3.10]{savlic}. Thus,
\begin{align*}
\textstyle{\geoE_i'} &\textstyle{= (-1)^{\cc_i} \sum_{k \in \Z} \left(1^{\otimes i-1} \otimes c_{\tnv}(\K_{\cc_i,\dee_i}(1,\nn_i,(\mm^k)_i)) \otimes 1^{\otimes s-i}\right) \times} \\ 
& \textstyle{\qquad \qquad \qquad \left(1^{\otimes i} \otimes c_{\tnv}(\K_{\cc_{i+1},\dee_{i+1}}(1,\nn_{i+1},(\mm^k)_{i+1})) \otimes 1^{\otimes s-i-1} \right)} \\
&\textstyle{= \sum_{k \in \Z} \geoPsi_i((k+1)r_i) \geoPsi_{i+1}^*(k r_{i+1} + 1).}
\end{align*}
The remaining equalities can be proved in an analogous manner.
\end{proof}

For $k = 0,1, \dots, r$, we can write
\[\textstyle{ k = r_1 + \cdots + r_{\ell-1} + k',}\]
for unique $1 \le \ell \le s$ and $0 \le k' \le r_{\ell} - 1$. For all $k$ such that $k' \ne 0$, let
\[\textstyle{ \geoE_k := \geoE_{k'}^\ell, \quad \geoF_k := \geoF_{k'}^\ell, \quad \geoH_k := \geoH_{k'}^\ell.}\]
For all $k$ such that $\ell \ne 1$ and $k'=0$, let
\[\textstyle{ \geoE_k := \geoE'_{\ell-1}, \quad \geoF_k := \geoF'_{\ell-1}, \quad \geoH_k := \geoH'_{\ell-1}.}\]
For $k=0$,
\[\textstyle{ \geoE_0 := \geoE'_0, \quad \geoF_0 := \geoF'_0, \quad \geoH_0 = \geoH'_0.}\]

\begin{thm}\label{thm:main}
For $k = 0,1,\dots,s-1$, and $n \in \Z-\{0\}$, the mapping
\[\textstyle{ E_k \mapsto \geoE_k, \quad F_k \mapsto \geoF_k, \quad H_k \mapsto \geoH_k, \quad I \otimes t^n \mapsto |n| \sum_{\ell=1}^s r_\ell \geoP_\ell(n r_\ell), \quad I \otimes 1 \mapsto \sum_{\ell=1}^s \geoP_\ell(0),}\]
defines a representation of $\widehat\gl_r$ on the space $\A$ and the linear map determined by $[\II] \mapsto \II$ is an isomorphism of $\widehat\gl_r$-representations $\A \to \FF$. This isomorphism maps $\A(0) \to \FF(0)$, and thus $\A(0) \cong V_\textup{basic}$.
\end{thm}

\begin{proof}
The theorem follows by Theorem \ref{thm:geoClifford}, Theorem \ref{thm:geoHeisenberg} and by comparing the formulas in Theorem \ref{thm:geoE,geoF} and Lemma \ref{lem:offdiag} to those in Lemma \ref{lem:glaction}.
\end{proof}

\begin{small}
\begin{ack}
I would like to thank Alistair Savage for his invaluable help throughout the writing of this paper. I would also like to thank Yuly Billig for his helpful advice.
\end{ack}
\end{small}

\bibliographystyle{plain}

\bibliography{docbib}

\end{document}